\documentclass[11pt, reqno]{amsart}
\usepackage{amssymb}
\usepackage{graphicx}
\usepackage{xcolor} 
\usepackage{tensor}
\usepackage{hyperref}
\usepackage[margin=1in]{geometry}

\definecolor{green}{rgb}{0,0.8,0} 

\newtheorem{theorem}{Theorem}[section]

\newtheorem{lemma}[theorem]{Lemma}
\newtheorem{proposition}[theorem]{Proposition}
\theoremstyle{definition}
\newtheorem{definition}[theorem]{Definition}

\theoremstyle{remark}

\numberwithin{equation}{section}
\newcommand{\nrm}[1]{\Vert#1\Vert}
\newcommand{\abs}[1]{\vert#1\vert}

\newcommand{\set}[1]{\{#1\}}

\newcommand{\br}[1]{\overline{#1}}

\renewcommand{\Re}{\mathrm{Re}}

\newcommand{\aeq}{\sim}
\newcommand{\aleq}{\lesssim}

\newcommand{\lap}{\triangle}

\newcommand{\ud}{\mathrm{d}}
\newcommand{\rd}{\partial}

\newcommand{\bb}{\Big}

\newcommand{\dlt}{\delta}

\newcommand{\eps}{\epsilon}
\newcommand{\veps}{\varepsilon}

\newcommand{\sgm}{\sigma}

\newcommand{\bfD}{{\bf D}}

\newcommand{\bfJ}{{\bf J}}

\newcommand{\bbC}{\mathbb C}

\newcommand{\bbR}{\mathbb R}

\newcommand{\calF}{\mathcal F}

\newcommand{\calN}{\mathcal N}

\newcommand{\calR}{\mathcal R}
\newcommand{\calS}{\mathcal S}
\newcommand{\calT}{\mathcal T}

\newcommand{\covD}{\bfD}
\newcommand{\covJ}{\bfJ}

\newcommand{\e}{\varepsilon}

\def\what{\widehat}

\begin{document}

\title[]{Decay and Scattering for the \\ Chern-Simons-Schr\"odinger equations}
\author{Sung-Jin Oh}%
\address{Department of Mathematics, UC Berkeley, Berkeley, CA, 94720}%
\email{sjoh@math.berkeley.edu}%

\author{Fabio Pusateri}%
\address{Department of Mathematics, Princeton University, Princeton, NJ, 08544}%
\email{fabiop@math.princeton.edu}%

\thanks{The first author is a Miller Research Fellow, and thanks the Miller Institute for the support. 
The second author was supported in part by a Simons Postdoctoral Fellowship and NSF grant DMS 1265875}%



\begin{abstract}
We consider the Chern-Simons-Schr\"odinger model in $1+2$ dimensions, and prove scattering for small solutions of the
Cauchy problem in the Coulomb gauge.
This model is a gauge covariant Schr\"odinger equation, with a potential decaying like $r^{-1}$ at infinity.
To overcome the difficulties due to this long range decay we start by performing $L^2$-based estimates covariantly.
This gives favorable commutation identities so that only curvature terms, which decay faster than $r^{-1}$,
appear in our weighted energy estimates.
We then select the Coulomb gauge to reveal a genuinely cubic null structure, which allows us to show sharp decay by Fourier methods.
\end{abstract}

\maketitle

\section{Introduction}
In this paper, we investigate the long term behavior and asymptotics of small solutions to the \emph{Chern-Simons-Schr\"odinger} \eqref{eq:CSS} system, which is a non-relativistic gauge field theory on $\bbR^{1+2}$ taking the form
\begin{equation} \label{eq:CSS} \tag{CSS}
\left\{
\begin{aligned}
	\covD_{t} \phi =& i \covD_{\ell} \covD_{\ell} \phi + i g \abs{\phi}^{2} \phi, \\
	F_{12} = & - \frac{1}{2} \abs{\phi}^{2}, \\
	F_{01} = &  - \frac{i}{2} ( \phi \br{\covD_{2} \phi} - (\covD_{2} \phi) \br{\phi} ), \\
	F_{02} = & \frac{i}{2} ( \phi \br{\covD_{1} \phi} - (\covD_{1} \phi) \br{\phi} ).
\end{aligned}
\right.
\end{equation}
Here $\phi$ is a $\bbC$-valued function (Schr\"odinger field), $A = A_{0} \ud x^{0} + A_{1} \ud x^{1} + A_{2} \ud x^{2}$ is a real-valued 1-form (gauge potential), $F_{\mu\nu} = \rd_{\mu} A_{\nu} - \rd_{\nu} A_{\mu}$ (curvature 2-form), $\covD_{\mu} = \rd_{\mu} + i A_{\mu}$ (covariant derivative) and $g \in \bbC$.
Our main theorem demonstrates that sufficiently small initial data lead to a global unique solution on $\bbR^{1+2}$ which exhibits \emph{linear scattering} to a free Schr\"odinger field in the Coulomb gauge $(\rd_{1} A_{1} + \rd_{2} A_{2} = 0)$; we refer to Theorem \ref{thm:mainThm} for a precise statement.

The \eqref{eq:CSS} system was introduced by Jackiw-Pi \cite{JP-PRL,JP-PRD}, with an emphasis on its self-dual structure (for $g=1$) and existence of (multi-)vortex solitons. 
It serves as a basic model for Chern-Simons dynamics on the plane, which is used to analyze various planar phenomena, e.g. anyonic statistics, fractional quantum Hall effect and high $T_{c}$ superconductors. For a more thorough discussion on the physical relevance of \eqref{eq:CSS} and self-duality, we refer the reader to \cite{JP-Rev,Dunne,HZ,Yang} and the references therein.

Recently, there has been some work regarding \eqref{eq:CSS} from the mathematical side. Various authors have contributed to the problem of local well-posedness \cite{BDS,Huh3,LST}; currently, the best result is due to Liu-Smith-Tataru \cite{LST}, who proved 
local well-posedness for $\phi(0) \in H^{\veps}_{x}$ for any $\veps > 0$. In the case of focusing self-interaction potential $g > 0$, Berg\'e-de Bouard-Saut \cite{BDS} showed finite time blow-up via a virial argument, and an explicit example was given later by Huh \cite{Huh1}. 
Global existence of a weak solution (without uniqueness) with sub-threshold charge was also proved in \cite{BDS}. 
%
We also mention the interesting papers \cite{Demoulini1, Demoulini2} on a closely related system proposed by Manton \cite{Manton}.
On the other hand, to our best knowledge, there has been no prior work concerning the asymptotic behavior of global solutions to \eqref{eq:CSS}.

The major difficulty in studying the asymptotic behavior of solutions to \eqref{eq:CSS} is the apparent long range behavior of the potentials $A_{j}$.
To elaborate, we begin from the fact that \eqref{eq:CSS} is \emph{gauge invariant}, i.e., it is invariant under the gauge transform
\begin{equation*}
	A_{\mu} \to A_{\mu} + \rd_{\mu} \chi, \quad
	\phi \to e^{- i \chi} \phi
\end{equation*}
where $\chi$ is a real-valued function on $\bbR^{1+2}$. Due to this fact, the initial value problem for \eqref{eq:CSS} is well-posed only after fixing a specific representative of $(A_{\mu}, \phi)$, a procedure usually referred to as \emph{gauge fixing}. For concreteness of our discussion, we shall choose the \emph{Coulomb gauge}, which is a gauge defined by the condition 
\begin{equation*}
	\rd_{1} A_{1} + \rd_{2} A_{2} = 0.
\end{equation*}
%
Expanding out $\covD_{\mu} = \rd_{\mu} + i A_{\mu}$, \eqref{eq:CSS} in the Coulomb gauge leads to the equation
\begin{equation} \label{eq:toySch}
	(\rd_{t} - i \lap) \phi = -2 A_{1} \rd_{1} \phi - 2 A_{2} \rd_{2} \phi - i A_{0} \phi + \hbox{(Higher order terms)}.
\end{equation}
Let us focus on the nonlinear terms with the most derivatives on $\phi$, i.e., $- 2 A_{j} \rd_{j} \phi$. 
Under the Coulomb condition, $A_{j}$ satisfies the equation $\lap A_{j} = \frac{1}{2} \eps_{jk} \rd_{k} \abs{\phi}^{2}$, 
where $\eps_{jk}$ is the unique anti-symmetric 2-form such that $\eps_{12} = 1$.
Therefore, $A_j$ is  given by the Biot-Savart law
\begin{equation} \label{eq:BiotSavart}
	A_{j}(t,x) 
		= - \frac{1}{2} \eps_{jk} \rd_{k} (-\lap)^{-1} \abs{\phi(t,x)}^{2} 
		= \frac{1}{4 \pi} \eps_{jk} \int_{\bbR^{2}} \frac{(x - y)^{k}}{\abs{x-y}^{2}} \abs{\phi(t, y)}^{2} \, \ud y.
\end{equation}
Then, it is easy to see that the right-hand side of \eqref{eq:BiotSavart} has an $r^{-1}$ tail as $r \to \infty$ 
for any non-trivial solution $\phi$ to \eqref{eq:CSS}, hence $A_{j}$ is a \emph{long range potential}. 
In principle, such long range potentials can lead to complicated asymptotic behaviors for even arbitrarily small solutions, such as modified scattering \cite{HN,KP} or finite time blow-up \cite{John0, John1}. 

The preceding discussion applies to other gauges as well, since the divergence-free part of $A_{j}$ (according to the Hodge decomposition of 1-forms on $\bbR^{2}$) is always given by the formula \eqref{eq:BiotSavart}. 
Indeed, the $r^{-1}$ behavior of (a part of) $A_{j}$ is present in the work \cite{LST}, in which a non-Coulomb gauge (more precisely, the heat gauge $A_{t} = \rd_{\ell} A_{\ell}$) was used.

Nevertheless, in this paper we are able to establish global existence of solutions of \eqref{eq:CSS} 
for sufficiently small initial data, and prove \emph{linear scattering}\footnote{By linear scattering, we mean convergence of the solution $\phi$ to some free Schr\"odinger field $e^{i t \lap} f_{\infty}$ as $t \to \infty$ in an appropriate topology; see \eqref{scattL^2} in Theorem \ref{thm:mainThm} for a more precise statement.} 
of such solutions in the Coulomb gauge. Our proof relies on the following two main observations:

\begin{list}{\labelitemi}{\leftmargin=2em}

\item[1)]  The \eqref{eq:CSS} system does not exhibit any long range behavior when viewed \emph{covariantly}, i.e., when phrased in terms of $\covD_{t}, \covD_{j}$ and $F_{\mu \nu}$ as in \eqref{eq:CSS}. 
Using the covariant charge identity \eqref{eq:chargeId} and covariant commutators such as $\covD_{j}$, $\covJ_{j} := x_{j} + 2 i t \covD_{j}$, we can establish global apriori $L^{2}$ bounds under the assumption that $\phi$ exhibits the linear decay rate, i.e., $\abs{t}^{-1}$ as $t \to \pm \infty$. 
Unfortunately, such a method seems to fall short of retrieving the $\abs{t}^{-1}$ decay rate, and for this purpose we turn to our second observation:

\item[2)] The \eqref{eq:CSS} system exhibits a \emph{strong, genuinely cubic null structure} in the Coulomb gauge. 
Remarkably, by (and only by) considering all cubic\footnote{The terms $A_{j} \rd_{j} \phi$ and $A_{0} \phi$ are cubic in $\phi$, as $A_{\mu}$ satisfies a Poisson equation with quadratic (and higher) terms in $\phi$ as a source. See \eqref{eq:CSS-C} in \S \ref{subsec:mainThm}} terms $(-2 A_{j} \rd_{j} \phi - i A_{0} \phi)$ together, we reveal a very strong null structure which effectively cancels out the long range effect of $A_{\mu}$ for the purpose of establishing the desired decay rate of $\phi$. Combined with the apriori $L^{2}$ bounds obtained using covariant methods, we are able to conclude the $\abs{t}^{-1}$ decay via an analysis of \eqref{eq:toySch} in Fourier space.
\end{list}

Broadly speaking, our method may be understood as a mix of two different existing approaches to small data scattering: The use of the covariant charge identity and commutators is akin to the tensor-geometric approach of Christodoulou-Klainerman \cite{CK1,CK2}, whereas the cubic null structure is identified and utilized within the Fourier-analytic framework of the works on space-time resonances 
\cite{GMS1} and \cite{nullcondition}.
A more detailed description of our main ideas will be given in Section \ref{sec:mainIdeas}, after we state the main theorem in \S \ref{subsec:mainThm}.

\subsection{Statement of the main theorem} \label{subsec:mainThm}
In the Coulomb gauge $\rd_{1} A_{1} + \rd_{2} A_{2} = 0$, which is the setting of our main theorem, \eqref{eq:CSS} can be reformulated as
\begin{equation} \label{eq:CSS-C} \tag{CSS-Coulomb}
\left\{
\begin{aligned}
	\rd_{t} \phi - i \lap \phi =&  - i A_{0} \phi - 2 A_{\ell} \rd_{\ell} \phi  - i A_{\ell} A_{\ell} \phi + i g \abs{\phi}^{2} \phi, \\
	\lap A_{1} = & \frac{1}{2} \rd_{2} \abs{\phi}^{2}, \\
	\lap A_{2} = & -\frac{1}{2} \rd_{1} \abs{\phi}^{2}, \\
	\lap A_{0} = & \frac{i}{2} \rd_{1} ( \phi \br{\covD_{2} \phi} - (\covD_{2} \phi) \br{\phi} ) 
			- \frac{i}{2} \rd_{2} ( \phi \br{\covD_{1} \phi} - (\covD_{1} \phi) \br{\phi} ),
\end{aligned}
\right.
\end{equation}
where repeated indices are assumed to be summed, e.g., $- 2 A_{\ell} \rd_{\ell} \phi = - 2 \sum_{\ell=1}^{2} A_{\ell} \rd_{\ell} \phi$ etc. 

We define the initial data set for \eqref{eq:CSS-C} as follows.
\begin{definition}[Coulomb initial data set] \label{def:id}
We say that a pair $(a_{j}, \phi_{0})$ of a 1-form $a_{j}$ and a $\bbC$-valued function $\phi$ on $\bbR^{2}$ is a \emph{Coulomb initial data set} for \eqref{eq:CSS} if it satisfies 
\begin{align}
	\rd_{1} a_{1} + \rd_{2} a_{2} =& 0, \label{eq:id:Coulomb} \\
	\rd_{1} a_{2} - \rd_{2} a_{1} =& - \frac{1}{2} \abs{\phi_{0}}^{2}. \label{eq:id:constraint}
\end{align}
\end{definition}

The condition \eqref{eq:id:Coulomb} is the Coulomb gauge condition for the initial data, whereas \eqref{eq:id:constraint} is the \emph{constraint equation} imposed by the system \eqref{eq:CSS}. We remark that the div-curl system \eqref{eq:id:Coulomb}--\eqref{eq:id:constraint} is enough to uniquely specify $a_{j}$ (with a mild condition at infinity) in terms of the (gauge invariant) amplitude of $\phi_{0}$, at least under our regularity assumptions below.

We are now ready to state the main theorem of this paper.
\begin{theorem}[Main Theorem] \label{thm:mainThm}
Let $(A_{j}(0), \phi(0))$ be a Coulomb initial data set for \eqref{eq:CSS} satisfying the smallness assumption
\begin{equation} \label{eq:condition4id}
	\sum_{m=0}^{2} \nrm{\covD^{(m)} \phi(0)}_{L^{2}_{x}}
	+ \nrm{\abs{x} \, \phi(0)}_{L^{2}_{x}} + \nrm{\abs{x} \, \covD \phi(0)}_{L^{2}_{x}} + \nrm{\abs{x}^{2} \phi(0)}_{L^{2}_{x}}
	\leq \veps_{1} .
\end{equation}
Then, for sufficiently small $\veps_{1}$ there exists a unique global solution $\phi \in C_{t} (\bbR; H^{2}_{x})$ of \eqref{eq:CSS-C} such that
\begin{equation} \label{eq:mainThm:decay}
	\nrm{\phi(t)}_{L^{\infty}_{x}} \aleq \veps_{1} (1+\abs{t})^{-1}.
\end{equation} 
Moreover, for each sign $\pm$ and $0 \leq s < 2$, there exists $f_{\pm \infty} \in H^s_x$ such that
\begin{align}
 \label{scattL^2}
e^{-it\lap} \phi(t) \stackrel{t\rightarrow \pm \infty}{\longrightarrow} f_{\pm\infty} \quad \mbox{in} \quad H^s_{x}.
\end{align}
\end{theorem}


\subsection{Notations and conventions} \label{subsec:notations}

\begin{list}{\labelitemi}{\leftmargin=1.5em}

\item Unless otherwise specified, we adopt the convention of summing up all repeated indices, e.g., $\covD_{j} \covD_{j} = \covD_{1} \covD_{1} + \covD_{2} \covD_{2}$.
\item $\eps_{jk}$ denotes the anti-symmetric 2-form with $\eps_{12} = 1$
\item We denote the ordinary and covariant spatial gradients by $D = (\rd_{1}, \rd_{2})$ and $\covD = (\covD_{1}, \covD_{2})$, respectively.
	The $m$-fold ordinary and covariant spatial gradients will be denoted by $D^{(m)}$ and $\covD^{(m)}$, respectively.
\item We define the operators $\covJ=(\covJ_{1}, \covJ_{2})$ and $J = (J_{1}, J_{2})$, where
\begin{equation*}
	\covJ_{k} := x_{k} + 2it \covD_{k}, \quad
	J_{k} := x_{k} + 2 i t \rd_{k}
\end{equation*}
Like $\covD, D$, we denote the $m$-fold application of $\covJ, J$ 
by $\covJ^{(m)}$, $J^{(m)}$, respectively.
	
\item We adopt the convention $\calF(\psi)(\xi) = \widehat{\psi}(\xi) := {(2 \pi)}^{-1} \int_{\bbR^{2}} e^{- i x \cdot \xi} \psi(x) \, \ud x$, for the Fourier transform.
\item We denote the space of Schwartz functions on $\bbR^{2}$ by $\calS_{x}$.
\end{list}


\section{Main ideas} \label{sec:mainIdeas}
The purpose of this section is to provide a more detailed description of the main ideas of our proof of Theorem \ref{thm:mainThm}. 
In order to establish global existence and scattering for small solutions, 
the basic strategy is to prove estimates for Sobolev and weighted $L^{2}$ norms
of $\phi$, and a decay estimate with the same rate of a linear solution, i.e.
\begin{equation} \label{eq:mainIdeas:decay}
	\nrm{\phi(t)}_{L^{\infty}_{x}} \aleq \veps_{1} (1+\abs{t})^{-1}.
\end{equation}


\subsubsection*{Covariant charge estimates}
To avoid the long range effect of $A_{\mu}$ in a fixed gauge, we derive apriori $L^{2}$ bounds using covariant methods. Our basic tool is the simple covariant charge identity \eqref{eq:inhomCovSch}-\eqref{eq:chargeId}. 
To bound higher derivatives we commute the covariant Schr\"odinger equation with $\covD_{j}$, and for weighted $L^{2}$ bounds we commute with 
\begin{equation*}
	\covJ_k := x_{k} + 2 i t \covD_{k} ,
\end{equation*} 
which is the covariant version of the well-known operator $J_{k} = x_{k} + 2 i t \rd_{k}$ for the Schr\"odinger equation. Such commutations are rather straightforward to perform in the covariant setting, leading to nice (schematic) equations such as
\begin{align*}
(\covD_{t} - i \covD_{\ell} \covD^{\ell}) \covD_{j} \phi 
&= \phi \cdot \phi \cdot \covD \phi, 			
\\
(\covD_{t} - i \covD_{\ell} \covD^{\ell}) \covJ_{j} \phi 
&= \phi \cdot \phi \cdot \covJ \phi + t (\phi \cdot \phi \cdot \covD \phi).
\end{align*}
See \eqref{eq:covSch:phi}--\eqref{eq:covSch:JJphi} for the full list. 
Then, assuming the decay \eqref{eq:mainIdeas:decay}, we can bootstrap covariant $L^{2}$ bounds of the form
\begin{equation*} 
	\nrm{\covD^{(m)} \phi(t)}_{H^{2}_{x}} \aleq \veps_{1}, \quad 
	\nrm{\covJ^{(m)} \phi(t)}_{L^{2}_{x}} \aleq \veps_{1} \big( \log (2+\abs{t}) \big)^{m} \quad (m=0, 1,2),
\end{equation*}
for $\veps_{1}$ sufficiently small.

\subsubsection*{Transition of covariant bounds to gauge-dependent bounds}
The next step consists in deriving $L^{2}$ bounds in the Coulomb gauge, such as
\begin{equation*}
	\nrm{D^{(m)} \phi(t)}_{L^{2}_{x}} \aleq \veps_{1}, \quad 
	\nrm{J^{(m)} \phi(t)}_{L^{2}_{x}} \aleq \veps_{1} \big( \log (2+\abs{t}) \big)^{m} \quad (m=0, 1,2),
\end{equation*}
from the covariant $L^{2}$ bounds. 
The execution of this step depends on estimates obtained from the elliptic equation for $A_{\mu}$ in the Coulomb gauge.
To eventually obtain the desired final result, it only remains to retrieve \eqref{eq:mainIdeas:decay}.

\subsubsection*{Asymptotic analysis in the Coulomb gauge}
Combined with the apriori growth bound of the weighted $L^{2}$ norms of $J^{(m)} \phi(t)$, 
a standard lemma (Lemma \ref{lem:improvedKS}) reduces the proof of decay \eqref{eq:mainIdeas:decay} to establishing uniform boundedness of $\widehat{\phi}(t)$, i.e.,
\begin{equation} \label{eq:mainIdeas:bddPhiHat}
	\nrm{\widehat{\phi}(t)}_{L^{\infty}_{\xi}} \aleq \veps_{1}.
\end{equation}
To achieve this, we shall use the Fourier-analytic framework of Germain-Masmoudi-Shatah \cite{GMS1}: 
Defining $f := e^{-it \lap} \phi$, the Schr\"odinger equation in the Coulomb gauge becomes
\begin{equation*}
	\rd_{t} \widehat{f} = e^{i t \abs{\xi}^{2}} \calF[- 2 A_{j} \rd_{j} \phi - i A_{0} \phi + \hbox{(Higher order terms)}]
\end{equation*}
As $\nrm{\widehat{f}}_{L^{\infty}_{\xi}} = \nrm{\widehat{\phi}}_{L^{\infty}_{\xi}}$, 
we simply need to estimate the $L^{\infty}_{\xi}$ norm of the above right-hand side.
The contribution of the higher order terms are easily manageable, and we are only left to consider the cubic terms $- 2 A_{j} \rd_{j} \phi - i A_{0} \phi$.

\subsubsection*{The cubic null structure of \eqref{eq:CSS} in the Coulomb gauge}
We now describe the strong cubic null structure of these cubic terms in the Coulomb gauge, which is crucial to close the whole argument. 
According to the framework of space-time resonances \cite{GMS1,nullcondition}, a cubic expression in $f = e^{-it\lap}\phi$ of the form
\begin{equation} \label{eq:mainIdea:nf}
\int_{\bbR^{2} \times \bbR^{2}} e^{i t \varphi(\eta, \sgm, \xi)} a(\eta, \sgm, \xi) \cdot D_{\eta, \sgm} \varphi(\eta, \sgm, \xi) 
	\, \widehat{f}(t,\eta - \sgm) \widehat{\overline{f}}(t,\sgm) \widehat{f}(t,\xi-\eta) \, \ud \eta \ud \sgm
\end{equation}
for some coefficient $a(\eta, \sgm, \xi)$, is a null structure 
(the symbol of the interaction vanishes on the space resonant set). 
Indeed, there is a gain of a factor of $t^{-1}$ upon integrating by parts in $\eta$ and/or $\sgm$. 
This may be viewed as an alternative interpretation of the classical null condition due to Klainerman \cite{K0,K1}. 

It is not difficult to see that the contribution of each $-2A_{j} \rd_{j} \phi$ and $- i A_{0} \phi$ is a null structure of the form \eqref{eq:mainIdea:nf}. However, the long range effect of $A_{\mu}$ is still manifest, in the sense that the coefficient $a(\eta, \sgm, \xi)$ has a singularity of the form $\abs{\eta}^{-2}$ (from the inversion of $\lap$ in the equation for $A_{\mu}$). Nevertheless, looking at $- 2 A_{j} \rd_{j} \phi - i A_{0} \phi$ as a whole gives (after a change of variables, and up to an irrelevant constant)
\begin{equation*}
\begin{aligned}
\int_{\bbR^{2} \times \bbR^{2}} e^{i t \varphi(\xi,\eta,\sigma)} \Big( \ud_{\eta} \log \abs{\eta} \wedge \ud_{\eta} \varphi(\xi,\eta,\sigma) \Big)
	 \widehat{f}(t,\xi-\sigma) \widehat{\overline{f}}(t,\sigma+\eta-\xi) \widehat{f}(t,\xi-\eta) \, \ud \sigma \ud \eta 
\end{aligned}
\end{equation*}
where $\ud_{\eta} g \wedge \ud_{\eta} h$ is a differential-forms notation for $\rd_{\eta_{1}} g \rd_{\eta_{2}} h - \rd_{\eta_{2}} g \rd_{\eta_{1}} h$; see \S \ref{subsec:CSS-CinFourier} for details. This is the claimed \emph{strong cubic null structure}: Not only is the singularity at $\eta = 0$ milder $(\ud_{\eta} \log \abs{\eta} \aeq \abs{\eta}^{-1}$), but it vanishes completely when an extra $\ud_{\eta}$ falls on $\ud_{\eta} \log \abs{\eta}$ 
after integration by parts (this cancellation reduces exactly to the standard fact $\ud^{2} = 0$ regarding exterior differential). 
Exploiting this null structure, as well as using the apriori $L^{2}$ bounds established earlier, we finally obtain \eqref{eq:mainIdeas:bddPhiHat} 
(for $\veps_{1}$ small enough) and close the whole argument.

\subsection*{Structure of the paper}
In Section \ref{sec:mainThmRed}, we reduce the main theorem 
(Theorem \ref{thm:mainThm}) to three propositions in accordance to the main ideas sketched above: Propositions \ref{prop:chargeEst}, \ref{prop:transition} and \ref{prop:decay} concerning \emph{covariant charge estimates}, \emph{transition from covariant to gauge-dependent bounds}, and \emph{asymptotic analysis in the Coulomb gauge}, respectively. Then in Sections \ref{sec:chargeEst}, \ref{sec:transition} and \ref{sec:decay}, we give proofs of these propositions in order.

\section{Reduction of the main theorem} \label{sec:mainThmRed}
In this section, we reduce the proof of Theorem \ref{thm:mainThm} to establishing three statements, namely Propositions \ref{prop:chargeEst}, \ref{prop:transition} and \ref{prop:decay}. Before we state the propositions, we shall fix a terminology: By an \emph{$H^{2}$ solution to \eqref{eq:CSS} in the Coulomb gauge on $(-T, T)$}, we mean a solution $(A_{\mu}, \phi)$ to \eqref{eq:CSS-C} such that $\phi \in C_{t} ((-T, T); H^{2}_{x})$.

In the first proposition, we prove gauge covariant apriori $L^{2}$ bounds, under bootstrap assumptions which include the critical $L^{\infty}_{x}$ decay assumption \eqref{eq:btstrp:decay}. We remind the reader that $\covJ_{k} := x_{k} + 2it \covD_{k}$.
\begin{proposition}[Covariant charge estimates] \label{prop:chargeEst}
Let $(A_{\mu}, \phi)$ be an $H^{2}$ solution to \eqref{eq:CSS} in the Coulomb gauge on $(-T, T)$, which satisfies the initial data estimate \eqref{eq:condition4id} and obeys the bootstrap assumptions
\begin{align} 
	\nrm{\phi(t)}_{L^{2}_{x}} + \nrm{\covD \phi(t)}_{L^{2}_{x}} + \nrm{\covD^{(2)} \phi(t)}_{L^{2}_{x}} 
	\leq & B \veps_{1} \label{eq:btstrp:cov:1} \\
	\nrm{\covJ \phi(t)}_{L^{2}_{x}} + \nrm{\covJ \covD \phi(t)}_{L^{2}_{x}}
	\leq & B \veps_{1} \log (2+\abs{t}) \label{eq:btstrp:cov:2} \\
	\nrm{\covJ^{(2)} \phi(t)}_{L^{2}_{x}}
	\leq & B \veps_{1} (\log (2+\abs{t}) )^{2} \label{eq:btstrp:cov:3} \\
	\nrm{\phi(t)}_{L^{\infty}_{x}} \leq& B \veps_{1} (1+\abs{t})^{-1} \label{eq:btstrp:decay}
\end{align}
for $t \in (-T, T)$ and an absolute constant $B>0$. 
Then, for $B$ sufficiently large and $\veps_{1} > 0$ small enough, we improve \eqref{eq:btstrp:cov:1},\eqref{eq:btstrp:cov:2} and \eqref{eq:btstrp:cov:3} 
on $(-T, T)$, respectively, to
\begin{align} 
	\nrm{\phi(t)}_{L^{2}_{x}} + \nrm{\covD \phi(t)}_{L^{2}_{x}} + \nrm{\covD^{(2)} \phi(t)}_{L^{2}_{x}} 
	\leq & \frac{B}{200} \veps_{1}  \label{eq:btstrp:impCov:1}\\
	\nrm{\covJ \phi(t)}_{L^{2}_{x}} + \nrm{\covJ \covD \phi(t)}_{L^{2}_{x}}
	\leq & \frac{B}{200} \veps_{1} \log (2+\abs{t})  \label{eq:btstrp:impCov:2} \\
	\nrm{\covJ^{(2)} \phi(t)}_{L^{2}_{x}}
	\leq & \frac{B}{200} \veps_{1} \big(\log (2+\abs{t}) \big)^{2}  \label{eq:btstrp:impCov:3}.
\end{align}  
\end{proposition}

The second proposition is used to translate the gauge covariant bounds \eqref{eq:btstrp:impCov:1}--\eqref{eq:btstrp:impCov:3} to the corresponding bounds in the Coulomb gauge. Recall that $J_{k} := x_{k} + 2it \rd_{k}$.

\begin{proposition}[Transition from covariant to gauge-dependent bounds] \label{prop:transition}
Let $(A_{\mu}, \phi)$ be an $H^{2}$ solution to \eqref{eq:CSS} in the Coulomb gauge on $(-T, T)$, which obeys the bootstrap assumption \eqref{eq:btstrp:decay} and the improved covariant bounds \eqref{eq:btstrp:impCov:1}--\eqref{eq:btstrp:impCov:3}. Then for $\veps_{1} > 0$ sufficiently small compared to $B$, the following gauge-dependent bounds hold on $(-T, T)$:
\begin{align} 
	\nrm{\phi(t)}_{L^{2}_{x}} + \nrm{D \phi(t)}_{L^{2}_{x}} + \nrm{D^{(2)} \phi(t)}_{L^{2}_{x}} 
	\leq & \frac{B}{100} \veps_{1} \label{eq:btstrp:ord:1} \tag{\ref{eq:btstrp:impCov:1}$'$}\\
	\nrm{J \phi(t)}_{L^{2}_{x}} + \nrm{J D \phi(t)}_{L^{2}_{x}} 
	\leq & \frac{B}{100} \veps_{1} \log (2+\abs{t}) \label{eq:btstrp:ord:2} \tag{\ref{eq:btstrp:impCov:2}$'$}\\
	\nrm{J^{(2)} \phi(t)}_{L^{2}_{x}}
	\leq & \frac{B}{100} \veps_{1} \big( \log (2+\abs{t}) \big)^{2} \label{eq:btstrp:ord:3} \tag{\ref{eq:btstrp:impCov:3}$'$}.
\end{align}
\end{proposition}

Finally, in the third proposition, we improve the $L^{\infty}_{x}$ decay assumption. The argument takes place entirely in the Coulomb gauge, and relies crucially on the \emph{cubic null structure} of \eqref{eq:CSS} in this gauge.
\begin{proposition}[Asymptotic analysis in the Coulomb gauge] \label{prop:decay}
Let $(A_{\mu}, \phi)$ be an $H^{2}$ solution to \eqref{eq:CSS} in the Coulomb gauge on $(-T, T)$, which satisfies the initial data estimate \eqref{eq:condition4id}. Assume furthermore that $(A_{\mu}, \phi)$ obeys the bootstrap assumption \eqref{eq:btstrp:decay} and the gauge-dependent bounds \eqref{eq:btstrp:ord:1}--\eqref{eq:btstrp:ord:3} for $t \in (-T, T)$.
 
Then for $B$ sufficiently large and $\veps_{1} > 0$ small enough, we improve \eqref{eq:btstrp:decay} on $(-T, T)$ to
\begin{equation}
	\nrm{\phi(t)}_{L^{\infty}_{x}} \leq \frac{B}{10} \veps_{1} (1+\abs{t})^{-1}. \label{eq:btstrp:impDecay}
\end{equation}

Moreover, for each sign $\pm$, there exists $\widehat{f_{\pm\infty}} \in L^\infty_\xi$ such that
\begin{align}
 \label{scattL^infty}
e^{it{|\xi|}^2} \widehat{\phi}(t) \stackrel{t\rightarrow \pm \infty}{\longrightarrow} \widehat{f_{\pm \infty}} \quad \mbox{in} \quad L^\infty_\xi .
\end{align}
\end{proposition}
For a more precise statement regarding the scattering property \eqref{scattL^infty}, we refer the reader to Proposition \ref{prop:decay:2},
and in particular \eqref{scattest}.


Assuming Propositions \ref{prop:chargeEst}, \ref{prop:transition} and \ref{prop:decay} for the moment, we now prove Theorem \ref{thm:mainThm}. 
\begin{proof} [Proof of Theorem \ref{thm:mainThm}]
We begin with a few quick reductions.
Let $(A_{j}(0), \phi(0))$ be a Coulomb initial data set satisfying \eqref{eq:condition4id}. 
By the $H^{2}$ local well-posedness theorem due to Berg\'e-de~Bouard-Saut \cite{BDS}, there exists a unique solution $\phi$ to the initial value problem for \eqref{eq:CSS} in the Coulomb gauge on some $(-T, T)$ such that $\phi \in C_{t}((-T, T); H^{2}_{x})$. 
Note that, by Proposition \ref{prop:transition}, the covariant bounds \eqref{eq:btstrp:cov:1}--\eqref{eq:btstrp:cov:3} and \eqref{eq:btstrp:decay} imply the gauge-dependent estimates \eqref{eq:btstrp:ord:1}--\eqref{eq:btstrp:ord:3}; in particular, this implies that $\sup_{t \in (-T, T)} \nrm{\phi(t)}_{H^{2}_{x}} \aleq \veps_{1}$ for every $(-T, T)$, for $\veps_{1}$ sufficiently small. This, by the same LWP theorem, establishes the global existence of $\phi$, which is unique in $C_{t} (\bbR; H^{2}_{x})$. 

Therefore, to prove global existence and the decay rate \eqref{eq:mainThm:decay}, it suffices to prove \eqref{eq:btstrp:cov:1}--\eqref{eq:btstrp:cov:3} and \eqref{eq:btstrp:decay} by a bootstrap argument; as this is rather standard, we will only give a brief sketch. It is obvious to see that the bootstrap assumptions \eqref{eq:btstrp:cov:1}--\eqref{eq:btstrp:cov:3} and \eqref{eq:btstrp:decay} are satisfied for small $\abs{t}$, simply by continuity.
Next, for any $T > 0$, \eqref{eq:btstrp:cov:1}--\eqref{eq:btstrp:cov:3} on $(-T, T)$ are improved to \eqref{eq:btstrp:impCov:1}--\eqref{eq:btstrp:impCov:3} provided that $B > 0$ is chosen sufficiently large and $\veps_{1} > 0$ is small enough, thanks to Proposition \ref{prop:chargeEst}. 
Applying Propositions \ref{prop:transition} and \ref{prop:decay} in order, we improve \eqref{eq:btstrp:decay} to \eqref{eq:btstrp:impDecay} as well, choosing $B>0$ larger and $\veps_{1}$ smaller if necessary.
Thus, by a standard continuity argument, we conclude that \eqref{eq:btstrp:cov:1}--\eqref{eq:btstrp:cov:3} and \eqref{eq:btstrp:decay} hold for any $T > 0$, thereby finishing the proof of Theorem \ref{thm:mainThm}.

We are now only left to establish the scattering property \eqref{scattL^2}. By symmetry, it suffices to consider the case $t \to +\infty$. 
By the preceding bootstrap argument, especially the bound \eqref{eq:btstrp:ord:1}, the $H^{2}_{x}$ norm of $e^{- i t \lap} \phi(t)$ is uniformly bounded, i.e.,
\begin{equation} \label{eq:scatt:pf:2}
	\sup_{t \in \bbR} \nrm{e^{- i t \lap} \phi(t)}_{H^{2}_{x}} \aleq B \veps_{1}.
\end{equation}
By compactness in the weak-star topology, there exists $f'_{\infty} \in H^{2}_{x}$ and a sequence $t_{k} \in [0, \infty)$ with $t_{k} \to \infty$ such that, in particular,
\begin{equation*}
	e^{- i t \lap} \phi(t_{k}) \stackrel{k \to \infty}{\rightharpoonup} f'_{\infty} \quad \hbox{in the distribution sense}.
\end{equation*}
Then by \eqref{scattL^infty} and uniqueness of distributional limit, we conclude that $f_{\infty} = f'_{\infty} \in H^{2}_{x}$.

Let $M > 2$ be a parameter to be chosen below. 
Given any $0 \leq s < 2$ and $\psi$ such that $\psi \in H^{2}_{x}$ and $\widehat{\psi} \in L^{\infty}_{\xi}$, note that
\begin{align*}
\nrm{\psi}_{H^{s}_{x}} \leq & \nrm{(1+\abs{\xi})^{s} \psi}_{L^{2}_{\xi} (\abs{\xi} \geq M)}
  + \nrm{(1+\abs{\xi})^{s} \psi}_{L^{2}_{\xi} (\abs{\xi} \leq M)} 
\\
  \aleq & M^{s-2} \nrm{\psi}_{H^{2}_{x}} + M^{1+s} \nrm{\widehat{\psi}}_{L^{\infty}_{\xi}} \, .
\end{align*}
This implies $\psi \in H^{s}_{x}$ and, optimizing the choice of $M$, we obtain
\begin{equation*}
 \nrm{\psi}_{H^{s}_{x}} \aleq \nrm{\psi}^{\frac{s+1}{3}}_{H^{2}_{x}} \nrm{\widehat{\psi}}^{\frac{2 - s}{3}}_{L^{\infty}_{\xi}} \, .
\end{equation*}
Applying this with $\psi = f_{\infty}$, we first conclude that $f_{\infty} \in H^{s}_{x}$ for any $0 \leq s < 2$. Moreover, another application of the preceding inequality with $\psi = e^{-i t \lap} \phi(t) - f_{\infty}$, combined with \eqref{eq:scatt:pf:2}, $\nrm{f_{\infty}}_{H^{2}_{x}} < \infty$ and \eqref{scattL^infty}, allow us to conclude \eqref{scattL^2} 
as desired.
\end{proof}

\section{Covariant charge estimates}\label{sec:chargeEst}
In this section, we prove Proposition \ref{prop:chargeEst} via covariant techniques to avoid the long-range potentials $A_{0}$, $A_{j}$. In \S \ref{subsec:chargeId}, we formulate and prove the covariant charge estimate, which will be our basic tool. Then in \S \ref{subsec:commutator}, we derive various commutation formulae, which will be used later to derive the covariant Schr\"odinger equations satisfied by the various fields we are interested in, e.g., $\phi, \covD \phi, \covD^{(2)} \phi$. In \S \ref{subsec:GN}, we prove covariant versions of the Gagliardo-Nirenberg inequality. Then finally, in \S \ref{subsec:pfOfChargeEst}, we put everything together and give a proof of Proposition \ref{prop:chargeEst}.

\subsection{Covariant charge identity} \label{subsec:chargeId}
Consider an inhomogeneous covariant Schr\"odinger equation
\begin{equation} \label{eq:inhomCovSch}
	(\covD_{t} - i \covD_{\ell} \covD_{\ell} ) \psi = N.
\end{equation}
Multiplying the equation by $\overline{\psi}$, taking the real part and integrating by parts, we see that
\begin{equation} \label{eq:chargeId}
	\frac{1}{2} \int_{\set{t = T_{2}}} \abs{\psi}^{2} \, \ud x  - \frac{1}{2} \int_{\set{t = T_{1}}} \abs{\psi}^{2} \, \ud x = \iint_{(T_{1}, T_{2}) \times \bbR^{2}} \Re(N \overline{\psi}) \, \ud t \ud x.
\end{equation}
The identity \eqref{eq:chargeId}, which we call the \emph{covariant charge identity}, will be the basis for our proof of Proposition \ref{prop:chargeEst}. From \eqref{eq:chargeId}, the following lemma follows immediately.

\begin{lemma} \label{lem:covChargeEst}
Let $\psi$ be a solution to \eqref{eq:inhomCovSch} such that $\psi \in C_{t} L^{2}_{x}$. Then for all $t \geq 0$, we have
\begin{equation} \label{eq:covChargeEst}
	\nrm{\psi(t)}_{L^{2}_{x}} \lesssim \nrm{\psi(0)}_{L^{2}_{x}} + \int_{0}^{t} \nrm{N(t')}_{L^{2}_{x}} \, \ud t'.
\end{equation}
An analogous statement holds for $t \leq 0$.
\end{lemma}

\subsection{Commutation formulae} \label{subsec:commutator}
The following lemma is the key computation of this section, 
and gives formulae for commuting $\covD_{j}$, $\covJ_{j}$ and the covariant Schr\"odinger operator of \eqref{eq:CSS}. 
\begin{lemma} \label{lem:comm4Sch}
Let $(A_{\mu}, \phi)$ be a $H^{2}$ solution to \eqref{eq:CSS}. Then the following commutation formulae hold:
\begin{align}
	(\covD_{t} - i \covD_{\ell} \covD_{\ell}) \covD_{k} \psi 
	= & \covD_{k} (\covD_{t} - i \covD_{\ell} \covD_{\ell}) \psi
		+ \eps_{k \ell} \bb( \abs{\phi}^{2} \covD_{\ell} \psi + \phi \br{\covD_{\ell} \phi} \psi \bb), \label{eq:comm4Sch:1} \\
	(\covD_{t} - i \covD_{\ell} \covD_{\ell}) \covJ_{k} \psi 
	= & \covJ_{k} (\covD_{t} - i \covD_{\ell} \covD_{\ell}) \psi
		+ 2 i t \eps_{k \ell} \bb( \abs{\phi}^{2} \covD_{\ell} \psi + \phi \br{\covD_{\ell} \phi} \psi \bb). \label{eq:comm4Sch:2}
\end{align}
\end{lemma}
\begin{proof} 
We begin with \eqref{eq:comm4Sch:1}. By \eqref{eq:CSS}, we have
\begin{align*}
	\covD_{k} \covD_{t} \psi  -\covD_{t} \covD_{k} \psi
	=& -\frac{1}{2} \eps_{k \ell}( \phi \br{\covD_{\ell} \phi} - (\covD_{\ell} \phi) \br{\phi} ) \psi
\end{align*}
and
\begin{align*}
	\covD_{k} \covD_{\ell} \covD_{\ell} \psi  
	= & \covD_{\ell} \covD_{k} \covD_{\ell} \psi + i \tensor{F}{_{k}_{\ell}} \covD_{\ell} \psi \\
	= & \covD_{\ell} \covD_{\ell} \covD_{k} \psi + i \tensor{F}{_{k}_{\ell}} \covD_{\ell} \psi + i \covD_{\ell} (F_{k \ell} \psi) \\
	= & \covD_{\ell} \covD_{\ell} \covD_{k} \psi -i \eps_{k \ell} \bb( \abs{\phi}^{2} \covD_{\ell} \psi + \frac{1}{2} (\phi \br{\covD_{\ell} \phi} + \covD_{\ell} \phi \br{\phi}) \psi \bb).
\end{align*}
Thus
\begin{equation*}
	\covD_{k} ( \covD_{t} - i \covD_{\ell} \covD_{\ell}) \psi 
	= (\covD_{t}  - i \covD_{\ell} \covD_{\ell}) ( \covD_{k} \psi) - \eps_{k \ell} \bb( \abs{\phi}^{2} \covD_{\ell} \psi + \phi \br{\covD_{\ell} \phi} \psi \bb),
\end{equation*}
which, upon rearranging the terms, gives \eqref{eq:comm4Sch:1}.

Next, we compute the commutator arising form $\bfJ_{k}$ to prove \eqref{eq:comm4Sch:2}. Using \eqref{eq:comm4Sch:1} and $\covD_{t} (t \psi) = t \covD_{t} \psi + \psi$, we see that
\begin{equation*}
	2 i t \covD_{k} ( \covD_{t} - i \covD_{\ell} \covD_{\ell}) \psi 
	= (\covD_{t}  - i \covD_{\ell} \covD_{\ell}) ( 2 i t \covD_{k} \psi) - 2 i t \eps_{k \ell} \bb( \abs{\phi}^{2} \covD_{\ell} \psi + \phi \br{\covD_{\ell} \phi} \psi \bb) - 2 i \covD_{k} \psi.
\end{equation*}
On the other hand, we easily compute
\begin{align*}
	x_{k} (\covD_{t} - i \covD_{\ell} \covD_{\ell}) \psi
	 =& \covD_{t} (x_{k} \psi) - i \covD_{\ell} (x_{k} \covD_{\ell} \psi) + i \covD_{k} \psi \\
	 =& (\covD_{t} - i \covD_{\ell} \covD_{\ell})(x_{k}  \psi) + 2 i \covD_{k} \psi. 
\end{align*}
Adding these up and rearranging the terms, we obtain \eqref{eq:comm4Sch:2}. \qedhere
\end{proof}

Our next lemma gives a simple formula for the commutator between $\covJ_{j}$ and $\covD_{k}$.
\begin{lemma} \label{lem:comm4DJ}
Let $(A_{\mu}, \phi)$ be an $H^{2}$ solution to \eqref{eq:CSS}. Then the following commutation formula holds.
\begin{equation} 
	\covJ_{j} \covD_{k} \psi - \covD_{k} \covJ_{j} \psi = \dlt_{jk} \psi + t \eps_{jk} \abs{\phi}^{2} \psi.
\end{equation}
\end{lemma}
\begin{proof} 
We compute
\begin{align*}
\bfJ_{j} \covD_{k} \psi - \covD_{k} \bfJ_{j} \psi  
= & x_{j} \covD_{k} \psi - \covD_{k} (x_{j}  \psi) + 2 i t (\covD_{j} \covD_{k} \psi - \covD_{k} \covD_{j} \psi) 
= \dlt_{j k} \psi + t \eps_{j k} \abs{\phi}^{2} \psi. \qedhere
\end{align*}
\end{proof}

We end this subsection with a Leibniz rule for the cubic nonlinearity of the form $\psi \br{\psi} \psi$.
\begin{lemma} \label{lem:comm4cubic}
Let $(A_{\mu}, \phi)$ be an $H^{2}$ solution to \eqref{eq:CSS}. Then the following formulae hold:
\begin{align} 
	\covD_{k} (\psi_{1} \br{\psi_{2}} \psi_{3}) 
	=& 	(\covD_{k} \psi_{1}) \br{\psi_{2}} \psi_{3} 
		+ \psi_{1} \br{\covD_{k} \psi_{2}} \psi_{3} 
		+ \psi_{1} \br{\psi_{2}} \covD_{k} \psi_{3} \label{eq:comm4cubic:1} \\
	\covJ_{k} (\psi_{1} \br{\psi_{2}} \psi_{3}) 
	=& 	(\covJ_{k} \psi_{1}) \br{\psi_{2}} \psi_{3} 
		- \psi_{1} \br{\covJ_{k} \psi_{2}} \psi_{3} 
		+ \psi_{1} \br{\psi_{2}} \covJ_{k} \psi_{3} \label{eq:comm4cubic:2}
\end{align}
\end{lemma}
\begin{proof} 
We shall only give a proof of \eqref{eq:comm4cubic:2}; the other formula \eqref{eq:comm4cubic:1} can be proved similarly.
Decompose $\bfJ_{k} = (x_{k} - 2 t A_{k}) + 2 i t \rd_{k}$. As the first term is real, we have
\begin{equation*}
(x_{k} - 2 t A_{k}) \psi_{1} \overline{\psi_{2}} \psi_{3} 
= (x_{k} - 2 t A_{k}) \psi_{1} \overline{\psi_{2}} \psi_{3}  - \psi_{1} \overline{(x_{k} - 2 t A_{k}) \psi_{2}} \psi_{3} + \psi_{1} \overline{\psi_{2}} (x_{k} - 2 t A_{k}) \psi_{3}.
\end{equation*}
On the other hand, for the second term, by Leibniz's rule, we have
\begin{equation*}
2it \rd_{k} (\psi_{1} \overline{\psi_{2}} \psi_{3} )
= (2it \rd_{k} \psi_{1}) \overline{\psi_{2}} \psi_{3}  - \psi_{1} \overline{(2it \rd_{k} \psi_{2})} \psi_{3} + \psi_{1} \overline{\psi_{2}} (2it \rd_{k} \psi_{3}).
\end{equation*}
Adding these up, we obtain the lemma. \qedhere
\end{proof}

\subsection{Covariant Gagliardo-Nirenberg inequalities} \label{subsec:GN}
To deal with some of the error terms arising from commutation, we will need the following covariant version of the standard Gagliardo-Nirenberg inequality $\nrm{D \psi}_{L^{4}_{x}} \aleq \nrm{\psi}_{L^{\infty}_{x}}^{1/2} \nrm{\lap \psi}_{L^{2}_{x}}^{1/2}$. 

\begin{lemma} \label{lem:GN4D}
For $\psi \in \calS_{x}$ and $A_{j} \in \calS_{x}$, we have
\begin{equation}  \label{eq:GN4D}
	\nrm{\covD \psi}_{L^{4}_{x}} \aleq \nrm{\psi}^{1/2}_{L^{\infty}_{x}} \nrm{ \covD^{(2)} \psi}_{L^{2}_{x}}^{1/2}
\end{equation}
\end{lemma}
\begin{proof} 
Using the identity $\rd_{j} (\psi^{1} \br{\psi^{2}}) = \covD_{j} \psi^{1} \br{\psi^{2}} + \psi^{1} \br{\covD_{j} \psi^{2}}$ and integrating by parts, we obtain
\begin{align*}
	\sum_{j = 1,2} \nrm{\covD_{j} \psi}_{L^{4}_{x}}^{4}
	=& \sum_{j = 1,2} \int \covD_{j} \psi \br{\covD_{j} \psi} \covD_{j} \psi \br{\covD_{j} \psi} \, \ud x \\
	=& - \sum_{j = 1,2} \int \psi \br{\covD_{j} \covD_{j} \psi} \covD_{j} \psi \br{\covD_{j} \psi} \, \ud x
		- 2 \sum_{j = 1,2} \int \psi \br{\covD_{j} \psi} \Re(\covD_{j} \psi \br{\covD_{j} \covD_{j} \psi}) \, \ud x.
\end{align*} 
Then using H\"older, we estimate the last line by
\begin{equation*}
	\aleq \nrm{\psi}_{L^{\infty}_{x}} \nrm{\covD^{(2)} \psi}_{L^{2}_{x}} ( \sum_{j=1,2} \nrm{\covD_{j} \psi}_{L^{4}_{x}}^{4} )^{1/2},
\end{equation*}
from which \eqref{eq:GN4D} follows. \qedhere
\end{proof}

We also need a Gagliardo-Nirenberg-type inequality for $\bfJ$.
\begin{lemma} \label{lem:GN4J}
For $\psi \in \calS_{x}$ and $A_{j} \in \calS_{x}$, we have
\begin{equation} 
\label{estlem:GN4J}
	\nrm{\covJ \psi}_{L^{4}_{x}} \aleq \nrm{\psi}^{1/2}_{L^{\infty}_{x}} \nrm{ \covJ^{(2)} \psi}_{L^{2}_{x}}^{1/2} 
\end{equation}
\end{lemma}

\begin{proof} 
Note the identities
\begin{equation*}
\bfJ_{j} 
	= 2 i t e^{i {|x|}^{2} / 4t} \covD_{j} e^{- i {|x|}^{2} / 4t}
,\quad
	\bfJ_{j} \bfJ_{k}	= -4 t^{2} e^{i {|x|}^{2} / 4t} \covD_{j} \covD_{k} e^{- i {|x|}^{2} / 4t}.
\end{equation*}
Thus, using Lemma \ref{lem:GN4D}, we estimate
\begin{align*}
\nrm{\covJ_{j} \psi}_{L^{4}_{x}}^{2}
= & \nrm{2 i t e^{i {|x|}^{2} / 4t} \covD_{j} e^{- i {|x|}^{2} / 4t} \psi }_{L^{4}_{x}}^{2}  \\
\aleq &\nrm{\psi}_{L^{\infty}_{x}} \nrm{4 t^{2} \covD^{(2)} e^{- i {|x|}^{2} / 4t} \psi}_{L^{2}_{x}} 
= \nrm{\psi}_{L^{\infty}_{x}} \nrm{\covJ^{(2)} \psi}_{L^{2}_{x}}. \qedhere
\end{align*}
\end{proof}

\subsection{Proof of Proposition \ref{prop:chargeEst}} \label{subsec:pfOfChargeEst}
Using the lemmas proved so far, it is not difficult to prove Proposition \ref{prop:chargeEst}. 

\begin{proof} [Proof of Proposition \ref{prop:chargeEst}]
We restrict our attention to $t \geq 0$; the other case is symmetric.
Furthermore, for notational simplicity, we introduce the following convention: We write $\psi_{1} \cdot \psi_{2} \cdot \psi_{3}$ for a linear combination of products of either $\psi_{j}$ or $\br{\psi_{j}}$ for $j=1,2,3$. If $\psi_{j}$ is vector-valued (e.g. $\covD \phi$), then it may take any of its components or the corresponding complex conjugate. The constants may depend on $g \in \bbC$. 

From Lemmas \ref{lem:comm4Sch}, \ref{lem:comm4cubic} and \eqref{eq:CSS}, it is not difficult to derive the following schematic equations\footnote{We remark that the particular structure $\psi_{1} \overline{\psi_{2}} \psi_{3}$ of the cubic nonlinearities arising from Lemma \ref{lem:comm4Sch} is important for applying Lemma \ref{lem:comm4cubic} to derive these equations. However, as it is not needed for applying the charge estimate after the equations are derived, we throw it away for notational simplicity.}:
\begin{align}
(\covD_{t} - i \covD_{\ell} \covD^{\ell}) \phi 
=& \phi \cdot \phi \cdot \phi, 							\label{eq:covSch:phi}\\
(\covD_{t} - i \covD_{\ell} \covD^{\ell}) \covD_{j} \phi 
=& \phi \cdot \phi \cdot \covD \phi, 						\label{eq:covSch:Dphi} \\
(\covD_{t} - i \covD_{\ell} \covD^{\ell}) \covD_{j} \covD_{k} \phi 
=& \phi \cdot \phi \cdot \covD^{(2)} \phi + \phi \cdot \covD \phi \cdot \covD \phi, \label{eq:covSch:DDphi}\\
(\covD_{t} - i \covD_{\ell} \covD^{\ell}) \covJ_{j} \phi 
=& \phi \cdot \phi \cdot \covJ \phi + t (\phi \cdot \phi \cdot \covD \phi), 	\label{eq:covSch:Jphi}\\
(\covD_{t} - i \covD_{\ell} \covD^{\ell}) \covJ_{j} \covD_{k} \phi 
=& \phi \cdot \phi \cdot \covJ \covD \phi + \phi \cdot \covD \phi \cdot \covJ \phi \label{eq:covSch:JDphi} \\
	& + t (\phi \cdot \phi \cdot \covD^{(2)} \phi + \phi \cdot \covD \phi \cdot \covD \phi), \notag \\
(\covD_{t} - i \covD_{\ell} \covD^{\ell}) \covJ_{j} \covJ_{k} \phi 
=& \phi \cdot \phi \cdot \covJ^{(2)} \phi + \phi \cdot \covJ \phi \cdot \covJ \phi \label{eq:covSch:JJphi} \\
	& + t (\phi \cdot \phi \cdot \covJ \covD \phi + \phi \cdot \phi \cdot \covD \covJ \phi + \phi \cdot \covD \phi \cdot \covJ \phi). \notag
\end{align}

We now claim that
\begin{align}
\nrm{(\covD_{t} - i \covD_{\ell} \covD^{\ell}) \covJ \phi (t)}_{L^{2}_{x}}
\aleq & (1+t)^{-1}B^{3} \veps_{1}^{3} \label{eq:chargeEst:pf:1}\\
\nrm{(\covD_{t} - i \covD_{\ell} \covD^{\ell}) \covJ \covD \phi (t)}_{L^{2}_{x}}
\aleq & (1+t)^{-1} B^{3} \veps_{1}^{3} \label{eq:chargeEst:pf:2}\\
\nrm{(\covD_{t} - i \covD_{\ell} \covD^{\ell}) \covJ^{(2)} \phi (t)}_{L^{2}_{x}}
\aleq & (1+t)^{-1} \log(2+t) B^{3} \veps_{1}^{3}. \label{eq:chargeEst:pf:3}
\end{align}
Indeed, using H\"older's inequality and Lemmas \ref{lem:GN4D}, \ref{lem:GN4J}, we can estimate:
\begin{align*}
\nrm{(\covD_{t} - i \covD_{\ell} \covD^{\ell}) \covJ \phi (t)}_{L^{2}_{x}}
\aleq & \nrm{\phi}_{L^{\infty}_{x}}^{2} \nrm{\covJ \phi}_{L^{2}_{x}} + t \nrm{\phi}_{L^{\infty}_{x}}^{2} \nrm{\covD \phi}_{L^{2}_{x}} \\
\aleq & B^{3} \veps_{1}^{3} (1+t)^{-2} \log(2+t) + B^{3} \veps_{1}^{3} \, t (1+t)^{-2} \\
\aleq & B^{3} \veps_{1}^{3} (1+t)^{-1}, \\
\nrm{(\covD_{t} - i \covD_{\ell} \covD^{\ell}) \covJ \covD \phi (t)}_{L^{2}_{x}}
\aleq & \nrm{\phi}_{L^{\infty}_{x}}^{2} \nrm{\covJ \covD \phi}_{L^{2}_{x}} + \nrm{\phi}_{L^{\infty}_{x}} \nrm{\covD \phi}_{L^{4}_{x}} \nrm{\covJ \phi}_{L^{4}_{x}} \\
	& + t \nrm{\phi}_{L^{\infty}_{x}}^{2} \nrm{\covD^{(2)} \phi}_{L^{2}_{x}} + t \nrm{\phi}_{L^{\infty}_{x}} \nrm{\covD \phi}^{2}_{L^{4}_{x}} \\
\aleq & B^{3} \veps_{1}^{3} (1+t)^{-2} \log(2+t) + B^{3} \veps_{1}^{3}\, t (1+t)^{-2} \\
\aleq & B^{3} \veps_{1}^{3} (1+t)^{-1}, \\
\nrm{(\covD_{t} - i \covD_{\ell} \covD^{\ell}) \covJ^{(2)}\phi (t)}_{L^{2}_{x}}
\aleq & \nrm{\phi}_{L^{\infty}_{x}}^{2} \nrm{\covJ^{(2)} \phi}_{L^{2}_{x}} + \nrm{\phi}_{L^{\infty}_{x}} \nrm{\covJ \phi}^{2}_{L^{4}_{x}} \\
	& + t \nrm{\phi}_{L^{\infty}_{x}}^{2} \nrm{\covJ \covD \phi}_{L^{2}_{x}} 
	   + t \nrm{\phi}_{L^{\infty}_{x}}^{2} \nrm{\covD \covJ \phi}_{L^{2}_{x}} \\
	&   + t \nrm{\phi}_{L^{\infty}_{x}} \nrm{\covD \phi}_{L^{4}_{x}} \nrm{\covJ \phi}_{L^{4}_{x}} \\
\aleq & B^{3} \veps_{1}^{3} (1+t)^{-2} \big( \log(2+t) \big)^{2} + B^{3} \veps_{1}^{3} \, t (1+t)^{-2} \log(2+t) \\
\aleq & B^{3} \veps_{1}^{3} (1+t)^{-1} \log(2+t).
\end{align*}
Note that we have used $\nrm{\covD \covJ \phi}_{L^{2}_{x}} \aleq B \veps_{1} \log(2+t)$, which follows immediately from Lemma \ref{lem:comm4DJ}.

Proceeding similarly, it is easy to also establish
\begin{equation}
\nrm{(\covD_{t} - i \covD_{\ell} \covD^{\ell}) \covD^{(m)} \phi (t)}_{L^{2}_{x}}
\aleq (1+t)^{-2} B^{3} \veps_{1}^{3} \label{eq:chargeEst:pf:4} 
\end{equation}
for $m = 0,1,2$. 
Then from \eqref{eq:chargeEst:pf:1}--\eqref{eq:chargeEst:pf:4}, Proposition \ref{prop:chargeEst} follows by an application of the charge estimate \eqref{eq:covChargeEst}. \qedhere
\end{proof}

\section{From covariant to gauge-dependent bounds}\label{sec:transition}
In this brief section, we prove Proposition \ref{prop:transition} concerning the transition from the covariant estimates \eqref{eq:btstrp:impCov:1}--\eqref{eq:btstrp:impCov:3} to the gauge-dependent estimates \eqref{eq:btstrp:ord:1}--\eqref{eq:btstrp:ord:3} in the Coulomb gauge. Our basic tool is the following set of estimates for the Schr\"odinger field $\phi$ and gauge potential $A_{j}$ in the Coulomb gauge. 

\begin{lemma}[Estimates in Coulomb gauge] \label{lem:est4A}
Let $(A_{\mu}, \phi)$ be an $H^{2}$ solution to \eqref{eq:CSS} in the Coulomb gauge on $(-T, T)$, which obeys \eqref{eq:btstrp:decay} and \eqref{eq:btstrp:impCov:1}. Then the following bounds hold for $t \in (-T, T)$:
\begin{align} 
	\nrm{\phi(t)}_{L^{p}_{x}} &\leq B \veps_{1} (1+\abs{t})^{-1+2/p} \quad \hbox{ for } 2 \leq p \leq \infty,		\label{eq:LpEst4phi} \\
	\nrm{A_{j}(t)}_{L^{p}_{x}} & \aleq_{p} B^{2} \veps_{1}^{2} (1+\abs{t})^{-1+2/p}  \quad \hbox{ for } 2 < p \leq \infty, \label{eq:est4A} \\ 
	\nrm{D A_{j}(t)}_{L^{p}_{x}} & \aleq_{p} B^{2} \veps_{1}^{2} (1+\abs{t})^{-2+2/p} \quad \hbox{ for } 2 \leq p < \infty. \label{eq:est4DA}
\end{align}
\end{lemma}
\begin{proof} 
The first estimate \eqref{eq:LpEst4phi} is an immediate consequence of interpolation between the $L^{\infty}_{x}$ and $L^{2}_{x}$ bound on $\phi$ in in \eqref{eq:btstrp:decay} and \eqref{eq:btstrp:impCov:1}, respectively. Next, in order to prove estimates for $A_{j}$, recall from \eqref{eq:CSS-C} that $A_{j}$ satisfies the following elliptic equation in the Coulomb gauge:
\begin{equation*}
	- \lap A_{j} = \frac{1}{2} \eps_{jk} \rd_{k} \abs{\phi}^{2}.
\end{equation*}
Thus, $\rd_{\ell} A_{j} = (\eps_{jk}/2) R_{\ell} R_{k} \abs{\phi}^{2}$, where $R_{j} = \rd_{j}/\sqrt{-\lap}$ is the Riesz transform.  
By the $L^{p}$ boundedness of the Riesz transform, we have for $1 < p < \infty$
\begin{equation*}
	\nrm{\rd_{\ell} A_{j}}_{L^{p}_{x}} \aleq_{p} \nrm{\phi}^{2}_{L^{2p}_{x}}.
\end{equation*}

On the other hand, by Hardy-Littlewood-Sobolev fractional integration, we have for $2 < p < \infty$
\begin{equation*}
	\nrm{A_{j}}_{L^{p}_{x}} \aleq_{p} \nrm{\phi}^{2}_{L^{4p / (2+p)}_{x}}.
\end{equation*}
Thus, the desired estimates \eqref{eq:est4A} and \eqref{eq:est4DA} for $p < \infty$ are an easy consequence of  \eqref{eq:LpEst4phi}. 
On the other hand, the case $p = \infty$ of \eqref{eq:est4A} follows from the Gagliardo-Nirenberg inequality $\nrm{A_{j}}_{L^{\infty}_{x}} \aleq \nrm{A_{j}}_{L^{4}_{x}}^{1/2} \nrm{D A_{j}}_{L^{4}_{x}}^{1/2}$ and the case $p=4$ of \eqref{eq:est4A}, \eqref{eq:est4DA}. \qedhere
\end{proof}

\begin{proof} [Proof of Proposition \ref{prop:transition}]
For simplicity, we restrict to $t \geq 0$.  Without loss of generality, we may assume that $B \veps_{1} \leq 1$. 
Expanding out the covariant derivatives, we have
\begin{align*}
	\covD_{j} \phi
	=& \rd_{j} \phi + i A_{j} \phi, \\
	\covD_{j} \covD_{k} \phi
	=& \rd_{j} \rd_{k} \phi + i (\rd_{j} A_{k}) \phi + i A_{k} \rd_{j} \phi + i A_{j} \rd_{k} \phi - A_{j} A_{k} \phi.
\end{align*}
Using H\"older, \eqref{eq:btstrp:decay}, \eqref{eq:btstrp:impCov:1} and \eqref{eq:LpEst4phi}--\eqref{eq:est4DA} we obtain:
\begin{align*}
	\nrm{\covD \phi(t) - D \phi(t)}_{L^{2}_{x}} 
	& \aleq \nrm{A(t)}_{L^{\infty}_{x}} \nrm{\phi(t)}_{L^{2}_{x}}
	\aleq B^{3} \veps_{1}^{3} (1+t)^{-1}, \\
	\nrm{\covD^{(2)} \phi(t) - D^{(2)} \phi(t)}_{L^{2}_{x}} 
	& \aleq \nrm{D A(t)}_{L^{2}_{x}} \nrm{\phi(t)}_{L^{\infty}_x}
		+ \nrm{A(t)}_{L^{\infty}_{x}} \nrm{D \phi(t)}_{L^{2}_{x}} 
		+ \nrm{A(t)}_{L^{\infty}_{x}}^{2} \nrm{\phi(t)}_{L^{2}_{x}} \\
	& \aleq B^{3} \veps_{1}^{3} (1+t)^{-1},										
\end{align*}
where on the last line, we additionally used $B \veps_{1} \leq 1$ and the estimate for $\nrm{D \phi(t) - \covD \phi(t)}_{L^{2}_{x}}$ from the first line to estimate $\nrm{D \phi(t)}_{L^{2}_{x}}$.

Similarly, expanding out $\covJ$ and $\covD$, we have
\begin{align*}
	\covJ_{j} \phi =& J_{j} \phi - 2 t A_{j} \phi, \\
	\covJ_{j} \covD_{k} \phi 
	=& J_{j} \rd_{k} \phi + i  A_{k} J_{j} \phi - 2 t (\rd_{j} A_{k}) \phi - 2 t A_{j} \rd_{k} \phi - 2 i t A_{j} A_{k} \phi, \\
	\bfJ_{j} \bfJ_{k} \phi
	= & J_{j} J_{k} \phi - 2 t A_{k} J_{j} \phi - 4 i t^{2} (\rd_{j} A_{k}) \phi
		 - 2 t A_{j} J_{k} \phi + 4 t^{2} A_{j} A_{k} \phi.
\end{align*}
Then, as before, we estimate via H\"older, \eqref{eq:btstrp:decay}--\eqref{eq:btstrp:impCov:3} and \eqref{eq:LpEst4phi}--\eqref{eq:est4DA}:
\begin{align*}
	\nrm{\covJ \phi(t) - J \phi(t)}_{L^{2}_{x}} 
	& \aleq 	t \nrm{A(t)}_{L^{\infty}_{x}} \nrm{\phi(t)}_{L^{2}_{x}}
	\aleq		 B^{3} \veps_{1}^{3} , \\
	\nrm{\covJ \covD \phi(t) - JD \phi(t)}_{L^{2}_{x}}
	& \aleq \nrm{A(t)}_{L^{\infty}_{x}} \nrm{J \phi(t)}_{L^{2}_{x}} + t \nrm{D A(t)}_{L^{2}_{x}} \nrm{\phi(t)}_{L^{\infty}_{x}} \\
	& \phantom{\aleq} + t \nrm{A(t)}_{L^{\infty}_x} \nrm{D\phi(t)}_{L^{2}_{x}} + t \nrm{A(t)}^{2}_{L^{\infty}_{x}} \nrm{\phi(t)}_{L^{2}_{x}} \\
	& \aleq B^{3} \veps_{1}^{3}, \\
	\nrm{\covJ^{(2)} \phi(t) - J^{(2)} \phi(t)}_{L^{2}_{x}}
	& \aleq t \nrm{A(t)}_{L^{\infty}_{x}} \nrm{J \phi(t)}_{L^{2}_{x}} + t^{2} \nrm{D A(t)}_{L^{2}_{x}} \nrm{\phi(t)}_{L^{\infty}_{x}}	\\
	& \phantom{\aleq} 
	+ t^{2} \nrm{A(t)}_{L^{\infty}_{x}}^2 \nrm{\phi(t)}_{L^{2}_{x}} \\
	& \aleq B^{3} \veps_1^{3} \log(2+t).
\end{align*}
Taking $\veps_{1} > 0$ sufficiently small 
Proposition \ref{prop:transition} follows.
\end{proof}

\section{Decay for the Schr\"odinger field}\label{sec:decay}
In this section, we prove Proposition \ref{prop:decay}, thereby completing the proof of Theorem \ref{thm:mainThm}.
In \S \ref{subsec:reductionOfdecay}, we reduce the proof of Proposition \ref{prop:decay} to establishing a uniform bound on $\nrm{\widehat{\phi}(t)}_{L^{\infty}_{\xi}}$; see Proposition \ref{prop:decay:2}. Then in \S \ref{subsec:CSS-CinFourier}, we rewrite the Schr\"odinger equation in the Coulomb gauge and reveal the \emph{cubic null structure} of \eqref{eq:CSS} in this gauge. Finally, in \S \ref{subsec:pfOfDecay}, we give a proof of Proposition \ref{prop:decay:2}.

\subsection{Reduction of Proposition \ref{prop:decay}} \label{subsec:reductionOfdecay}

The first step in the proof of the sharp $|t|^{-1}$ decay of $\phi$ is given by the following standard lemma: 
\begin{lemma} \label{lem:improvedKS}

For $\psi \in C_{t} \calS_{x}$ and $\abs{t} \geq 1$ we have
\begin{equation}
	\nrm{\psi(t)}_{L^{\infty}_{x}} \aleq \frac{1}{\abs{t}} \nrm{\widehat{\psi}(t)}_{L^{\infty}_{\xi}} + \frac{1}{\abs{t}^{5/4}} \sum_{m=0}^{2} \nrm{J^{(m)} \psi(t)}_{L^{2}_{x}}.
\end{equation}
\end{lemma}
For a proof of the above, we refer to \cite{HN}.
Thanks to this, one can easily see that establishing Proposition \ref{prop:decay} can be reduced to the following proposition: 

\begin{proposition} \label{prop:decay:2}
Let $(A_{\mu}, \phi)$ be an $H^{2}$ solution to \eqref{eq:CSS} in the Coulomb gauge which satisfies the initial data estimate
\begin{equation*} \tag{\ref{eq:condition4id}}
	\sum_{m=0}^{2} \nrm{\covD^{(m)} \phi(0)}_{L^{2}_{x}}
	+ \nrm{\abs{x} \, \phi(0)}_{L^{2}_{x}} + \nrm{\abs{x} \, \covD \phi(0)}_{L^{2}_{x}} + \nrm{\abs{x}^{2} \phi(0)}_{L^{2}_{x}}
	\leq \veps_{1}.
\end{equation*}
Assume furthermore that $(A_{\mu}, \phi)$ obeys the bootstrap assumption
\begin{equation*}
	\nrm{\phi(t)}_{L^{\infty}_{x}} \leq B \veps_{1} (1+\abs{t})^{-1} \tag{\ref{eq:btstrp:decay}} \\
\end{equation*}
and the improved bounds
\begin{align*} 
	\nrm{\phi(t)}_{L^{2}_{x}} + \nrm{D \phi(t)}_{L^{2}_{x}} + \nrm{D^{(2)} \phi(t)}_{L^{2}_{x}} 
	\leq & \frac{B}{100} \veps_{1} \tag{\ref{eq:btstrp:impCov:1}$'$}\\
	\nrm{J \phi(t)}_{L^{2}_{x}} + \nrm{J D \phi(t)}_{L^{2}_{x}} 
	\leq & \frac{B}{100} \veps_{1} \log (2+\abs{t})  \tag{\ref{eq:btstrp:impCov:2}$'$}\\
	\nrm{J^{(2)} \phi(t)}_{L^{2}_{x}}
	\leq & \frac{B}{100} \veps_{1} (\log (2+\abs{t}))^{2} \tag{\ref{eq:btstrp:impCov:3}$'$}
\end{align*}
for $t \geq 0$. 
Then, for any $0\leq t_1 \leq t_2$ we have
\begin{equation} \label{scattest}
{\big\| e^{it_2{|\xi|}^2} \widehat{\phi}(t_2,\xi) - e^{it_1{|\xi|}^2} \widehat{\phi}(t_1,\xi) \big\|}_{L^\infty_\xi} \lesssim B^3 \veps_1^3 {(1+t_1)}^{-1/10}.
\end{equation}
In particular, given $\dlt > 0$, choosing $B$ sufficiently large and $\veps_{1}$ small enough, we have
\begin{equation} \label{eq:decay:2}
	\nrm{\widehat{\phi}(t)}_{L^{\infty}_{\xi}} \leq \dlt B \veps_{1},
\end{equation}
and, moreover, there exists $\widehat{f_\infty} \in L^\infty_\xi$ such that
\begin{equation} \label{scattest2}
{\big\| e^{it{|\xi|}^2} \widehat{\phi}(t) - \widehat{f_\infty} \big\|}_{L^\infty_\xi} \lesssim B^3 \veps_1^3 {(1+t)}^{-1/10},
\end{equation}
for all $t\geq 0$. An analogous statement holds for $t \leq 0$.

\end{proposition}

In the rest of this section, we will be concerned with the proof of Proposition \ref{prop:decay:2}


%

\subsection{Cubic null structure in the Coulomb gauge} \label{subsec:CSS-CinFourier}
We first split $A_{0}$ into its quadratic and quartic parts, i.e., $A_{0} = A_{0,1} + A_{0, 2}$, where
\begin{align*}
	A_{0,1} =& \frac{i}{2} (-\lap)^{-1} \Big( - \rd_{1} (\phi \br{\rd_{2} \phi} - \rd_{2} \phi \br{\phi}) + \rd_{2} (\phi \br{\rd_{1} \phi} - \rd_{1} \phi \br{\phi}) \Big), \\
	A_{0,2} =& -(-\lap)^{-1} \Big(\rd_{1} (A_2 {|\phi|}^2) - \rd_{2} (A_1 {|\phi|}^2) \Big).
\end{align*}
Then we may write the Schr\"odinger equation in the Coulomb gauge as
\begin{equation*}
\rd_{t} \phi - i \lap \phi = \mathcal{N} + \mathcal{R} + \calT,
\end{equation*}
where
\begin{align*}
\mathcal{N} & :=
- i A_{0,1} \phi 
  - 2 A_{\ell} \rd_{\ell} \phi 
\\
\mathcal{R} & :=  - i A_{0,2} \phi
  - i A_{\ell} A_{\ell} \phi,
  \\
 \mathcal{T} & := i g \abs{\phi}^{2} \phi.
\end{align*}
In words, $\calN$ and $\calR$ are the cubic and quintic terms arising from the covariant Schr\"odinger operator $(\covD_{t} - i \covD^{\ell} \covD_{\ell})$, respectively, and $\calT$ is the cubic self-interaction term $g \abs{\phi}^{2} \phi$.
We may write $\calN$ and $\calR$ more explicitly as follows:
\begin{align*}
\calN = 
& 	(-\lap)^{-1} \big(-\rd_{1} \phi \br{\rd_{2} \phi} + \rd_{2} \phi \br{\rd_{1} \phi} \, \big) \phi \\
& 	+ (-\lap)^{-1} (\rd_{2} \abs{\phi}^{2} )\rd_{1} \phi 
	- (-\lap)^{-1} (\rd_{1} \abs{\phi}^{2}) \rd_{2} \phi, \\
\calR = &  i (-\lap)^{-1} \Big(\rd_{1} (A_2 {|\phi|}^2) - \rd_{2} (A_1 {|\phi|}^2) \Big) \phi 
  - i A_{\ell} A_{\ell} \phi.	
\end{align*}

Define $ f(t,x) := \big( e^{-it\lap} \phi(t) \big) (t,x)$. Then 
\begin{align*}
 \partial_t f(t) = e^{-it\lap} \big(  \mathcal{N}(t) + \mathcal{R}(t) + \calT(t) \big),
\end{align*}
and thus taking the Fourier transform,
\begin{equation} \label{eq:SchInFS}
\rd_{t} \widehat{f}(t) = e^{i t \abs{\xi}^{2}} \big( \widehat{\calN}(t) + \widehat{\calR}(t) + \widehat{\calT}(t) \big).
\end{equation}
Then, in order to estimate $|\widehat{\phi}(t)| = |\widehat{f}(t)|$, we estimate the right-hand side of \eqref{eq:SchInFS}, viz. $\widehat{\calN}$, $\widehat{\calR}$ and $\widehat{\calT}$ in $L^{\infty}_{\xi}$. 
With this goal in mind, we shall now demonstrate the \emph{cubic null structure} of $\calN$. 
We start by writing $\mathcal{N}$ in the Fourier space as follows:
\begin{align*}
\begin{split}
\widehat{\mathcal{N}}(t,\xi)
= & \frac{1}{4\pi^{2}} \int_{\bbR^2 \times \bbR^2} {|\eta|}^{-2}\big[ (\eta_1-\sigma_1)\sigma_2 - (\eta_2-\sigma_2)\sigma_1 \big] 
  \widehat{\phi}(t,\eta-\sigma) \widehat{\overline{\phi}}(t,\sigma) \widehat{\phi}(t,\xi-\eta) \, \ud \sigma \ud \eta
\\
& + \frac{1}{4\pi^{2}} \int_{\bbR^2 \times \bbR^2} {|\eta|}^{-2}\big[ -\eta_2(\xi_1-\eta_1) +\eta_1(\xi_2-\eta_2) \big] 
  \widehat{\phi}(t,\eta-\sigma) \widehat{\overline{\phi}}(t,\sigma) \widehat{\phi}(t,\xi-\eta) \, \ud \sigma \ud \eta 
\\
= & \frac{1}{4\pi^{2}} \int_{\bbR^2 \times \bbR^2} {|\eta|}^{-2}\big[ (\xi_2+\sigma_2)\eta_1 - (\xi_1+\sigma_1)\eta_2 \big] 
  \widehat{\phi}(t,\eta-\sigma) \widehat{\overline{\phi}}(t,\sigma) \widehat{\phi}(t,\xi-\eta) \, \ud \sigma \ud \eta. 
\end{split}
\end{align*}
Next, we change variables ($\sigma \rightarrow \sigma+\eta-\xi$), and write the above expression in terms of $f$:
\begin{align*}
\begin{split}
4\pi^{2} \widehat{\mathcal{N}}(t,\xi)
  & = e^{-it{|\xi|^2}} \int_{\bbR^2 \times \bbR^2} e^{it \varphi(\eta,\sigma)} {|\eta|}^{-2} m(\eta,\sigma)
  \widehat{f}(t,\xi-\sigma) \widehat{\overline{f}}(t,\sigma+\eta-\xi) \widehat{f}(t,\xi-\eta) \, \ud \sigma \ud \eta,
\\
\mbox{where} & \quad \varphi(\eta,\sigma) := {|\xi|}^2 -{|\xi-\sigma|}^2 + {|\sigma+\eta-\xi|}^2 - {|\xi-\eta|}^2 = 2 \eta \cdot \sigma,
\\
& \quad m(\eta,\sigma) := \sigma_2 \eta_1 - \sigma_1 \eta_2 = \frac{1}{2} \big( \eta_1\partial_{\eta_2}\varphi - \eta_2 \partial_{\eta_1}\varphi \big).
\end{split}
\end{align*}
The identity relating $m$ and $\varphi$ above is a null structure and we can use it to integrate by parts in frequency. 
Indeed, using the identities 
\begin{equation*}
\rd_{\eta_{j}} (-\log \abs{\eta}) = - \eta_{j} \abs{\eta}^{-2}  \quad \mbox{and} \quad
\rd_{\eta_{j}} e^{it \varphi(\eta, \sgm)} = i t \rd_{\eta_{j}} \varphi(\eta,\sigma) \, e^{it \varphi(\eta, \sgm)},
\end{equation*}
we see that $e^{i t \abs{\xi}^{2}} \widehat{\calN}(t, \xi)$ is given by
\begin{equation} \label{eq:nullStr4N} 
\begin{aligned}
- \frac{1}{8\pi^{2} i t} 
\int_{\bbR^{2} \times \bbR^{2}} \big( \ud_{\eta} (-\log \abs{\eta}) \wedge \ud_{\eta} e^{i t \varphi(\eta, \sgm)} \big) 
	 \widehat{f}(t,\xi-\sigma) \widehat{\overline{f}}(t,\sigma+\eta-\xi) \widehat{f}(t,\xi-\eta) \, \ud \sigma \ud \eta
\end{aligned}
\end{equation}
where $\ud_{\eta} f \wedge \ud_{\eta} g$ is a shorthand for $\rd_{\eta_{1}} f \rd_{\eta_{2}} g - \rd_{\eta_{2}} f \rd_{\eta_{1}} g$. 
Notice the crucial gain of a power of $t^{-1}$. 
Moreover, when $\ud_{\eta}$ is integrated by parts off from $e^{i t \varphi(\eta, \sgm)}$, the contribution of $\ud_{\eta} (-\log \abs{\eta})$ is zero, thanks to the fact that $\rd_{\eta_{1}} \rd_{\eta_{2}} (- \log \abs{\eta}) - \rd_{\eta_{2}} \rd_{\eta_{1}} (- \log \abs{\eta})  = 0$. This special cancellation is the aforementioned \emph{strong, genuinely cubic null structure} of $\calN$.

\subsection{Uniform boundedness of $\widehat{\phi}$ in the Coulomb gauge} \label{subsec:pfOfDecay}
In this subsection, we prove Proposition \ref{prop:decay:2}, which concludes the proof of Theorem \ref{thm:mainThm}. 

\begin{proof} [Proof of Proposition \ref{prop:decay:2}]
For simplicity, we again restrict to $t \geq 0$. Under the apriori assumptions \eqref{eq:btstrp:decay}, \eqref{eq:btstrp:ord:1}--\eqref{eq:btstrp:ord:3}, we claim that it suffices to show
\begin{align} 
\nrm{\widehat{\calN}(t)}_{L^{\infty}_{\xi}} \aleq & B^{3} \veps_{1}^{3} (1+t)^{-9/8} \label{boundhatNR} \\
\nrm{\widehat{\calR}(t)}_{L^{\infty}_{\xi}} \aleq & B^{5} \veps_{1}^{5} (1+t)^{-11/10} \label{boundhatR} \\
\nrm{\widehat{\calT}(t)}_{L^{\infty}_{\xi}} \aleq & B^{5} \veps_{1}^{5} (1+t)^{-11/10} \label{boundhatT}.
\end{align}
Indeed, integrating in $t$ the identity \eqref{eq:SchInFS}, the bounds \eqref{boundhatNR}-\eqref{boundhatT} immediately imply \eqref{scattest}.
Moreover since the initial data bound \eqref{eq:condition4id} implies
\begin{equation} \label{eq:decay:pf:1}
	\nrm{\widehat{\phi}(0)}_{L^{\infty}_{\xi}} \lesssim \nrm{\phi(0)}_{L^{1}_{x}} \aleq \nrm{(1+\abs{x}^{2}) \phi(0)}_{L^{2}_{x}} \aleq \veps_{1},
\end{equation}
we easily see how \eqref{eq:decay:2} follows.

Before we proceed to establish the claim, we point out a few consequences of our apriori assumptions which will be useful later. Thanks to the logarithmic growth in \eqref{eq:btstrp:ord:2}--\eqref{eq:btstrp:ord:3}, for any $p_{0} > 0$ we have
\begin{equation*} 
	\nrm{J \phi(t)}_{L^{2}_{x}} + \nrm{J D \phi(t)}_{L^{2}_{x}} + \nrm{J^{(2)} \phi(t)}_{L^{2}_{x}} \aleq_{p_{0}} B \veps_{1} (1+t)^{p_{0}}.
\end{equation*}
Since $x_{j}$ conjugates to $J_{j}$ via $e^{i t \lap}$ (i.e., $J_{j} \phi = e^{i t \lap} (x_{j} f)$), we have
\begin{equation*}
	\nrm{f(t)}_{L^{2}_{x}} + \nrm{\abs{x}^{2} f(t)}_{L^{2}_{x}} \aleq_{p_{0}} B \veps_{1} (1+t)^{p_{0}}.
\end{equation*}
Moreover, proceeding as in \eqref{eq:decay:pf:1}, we see that
\begin{equation*}
	\nrm{\widehat{\phi}(t)}_{L^{\infty}_{\xi}} 
	= \nrm{\widehat{f}(t)}_{L^{\infty}_{\xi}} 
	\aleq \nrm{(1+\abs{x}^{2}) f(t)}_{L^{2}_{x}} 
	\aleq_{p_{0}} B \veps_{1} (1+t)^{p_{0}}.
\end{equation*}

In what follows, we fix $0 < p_{0} < 1/20$.

\subsection{Estimate of the cubic null form $\mathcal{N}$}
Here, we shall prove \eqref{boundhatNR}.  We begin by integrating \eqref{eq:nullStr4N} by parts in $\eta_{1}$ and $\eta_{2}$.
Then, since
$\partial_{\eta_1} (\eta_2 |\eta|^{-2}) - \partial_{\eta_2} (\eta_1 |\eta|^{-2})=0$,
we can write
\begin{align}
\label{NIBP}
\begin{split}
e^{it{|\xi|^2}} \widehat{\mathcal{N}}(t,\xi) & = \frac{1}{8 \pi^{2} i t} \big( N_1(t,\xi) + N_2(t,\xi) + N_3(t,\xi) + N_4(t,\xi) \big),
\\
N_1(t,\xi) & := - \int_{\bbR^2 \times \bbR^2} e^{it \varphi(\eta,\sigma)} {|\eta|}^{-2}
  \eta_1 \widehat{f}(t,\xi-\sigma) \partial_{\eta_2} \widehat{\overline{f}}(t,\sigma+\eta-\xi)
  \widehat{f}(t,\xi-\eta) \, \ud \sigma \ud \eta,
\\
N_2(t,\xi) & := \int_{\bbR^2 \times \bbR^2} e^{it \varphi(\eta,\sigma)} {|\eta|}^{-2} \eta_2 \widehat{f}(t,\xi-\sigma) \partial_{\eta_1}  \widehat{\overline{f}}(t,\sigma+\eta-\xi) 
 \widehat{f}(t,\xi-\eta) \, \ud \sigma \ud \eta,
\\
N_3(t,\xi) & := - \int_{\bbR^2 \times \bbR^2} e^{it \varphi(\eta,\sigma)} {|\eta|}^{-2}\eta_1
  \widehat{f}(t,\xi-\sigma) \widehat{\overline{f}}(t,\sigma+\eta-\xi) 
  \partial_{\eta_2} \widehat{f}(t,\xi-\eta) \, \ud \sigma \ud \eta,
\\
N_4(t,\xi) & := \int_{\bbR^2 \times \bbR^2} e^{it \varphi(\eta,\sigma)} {|\eta|}^{-2}\eta_2
  \widehat{f}(t,\xi-\sigma) \widehat{\overline{f}}(t,\sigma+\eta-\xi) 
  \partial_{\eta_1}\widehat{f}(t,\xi-\eta) \, \ud \sigma \ud \eta.
\end{split}
\end{align}

Let us fix $\chi: [0,\infty) \rightarrow [0,1]$ a smooth function supported in $[0,2]$ and equal to $1$ in $[0,1]$.
For $M > 0$ we let $P_{\leq M}$ denote the projection operator defined by the Fourier multiplier $ \xi \rightarrow \chi(|\xi|M^{-1})$, 
i.e. $(\mathcal{F} P_{\leq M} f)(\xi) = \chi(|\xi|M^{-1}) \what{f}(\xi)$.
We split $N_1$ as $N_{1} = M_{1} + M_{2}$, where
\begin{align*}
\begin{split}
M_1(t,\xi) & := -\int e^{it \varphi(\eta,\sigma)} {|\eta|}^{-2} \chi(\eta {(1+t)}^{1/4})
  \eta_1 \widehat{f}(t,\xi-\sigma) \partial_{\eta_2}\widehat{\overline{f}}(t,\sigma+\eta-\xi) 
  \widehat{f}(t,\xi-\eta) \, \ud \sigma \ud \eta,
\\
M_2(t,\xi) & := \int e^{it \varphi(\eta,\sigma)} {|\eta|}^{-2}\big[ \chi(\eta {(1+t)}^{1/4})-1 \big]
  \eta_1  \widehat{f}(t,\xi-\sigma) \partial_{\eta_2}\widehat{\overline{f}}(t,\sigma+\eta-\xi)
  \widehat{f}(t,\xi-\eta) \, \ud \sigma \ud \eta.
\end{split}
\end{align*}

We then estimate
\begin{align*}
| M_1(t,\xi) | & \lesssim {\| \partial_2 \widehat{f}(t) \|}_{L^2} {\| \widehat{f}(t) \|}_{L^2} 
  \int {|\eta|}^{-1} \chi(\eta {(1+t)}^{1/4}) |\widehat{f}(t,\xi-\eta)|\,\ud \eta
\\
& \lesssim {\| \partial_2 \widehat{f}(t) \|}_{L^2} {\| \widehat{f}(t) \|}_{L^2} {\| \widehat{f}(t) \|}_{L^\infty} {(1+t)}^{-1/4}
\lesssim B^{3} \e_1^3 {(1+t)}^{-1/4+2p_0},
\end{align*}
and
\begin{align*}
| M_2(t,\xi) | & \lesssim {\Big\| P_{\geq {(1+t)}^{-1/4}} \partial_1\lap^{-1} \big( e^{-it\lap} x_2 \overline{f}(t) \phi(t) \big) \Big\|}_{L^2} 
  {\| f(t) \|}_{L^2} 
\\
& \lesssim  {(1+t)}^{1/4} {\| x_2 f(t) \|}_{L^2} {\|\phi(t)\|}_{L^\infty} {\| f(t) \|}_{L^2} \lesssim B^{3} \e_1^3 {(1+t)}^{-3/4+p_0}.
\end{align*}
The term $N_2$ can be estimated in the same way as $N_1$.

To estimate $N_3$ we perform the same splitting as above, but we will need slightly different estimates.
We write $N_3 = M_3 + M_4$, where
\begin{align*}
\begin{split}
M_3(t,\xi) & := - \int e^{it \varphi(\eta,\sigma)} {|\eta|}^{-2}\eta_1 \chi(\eta {(1+t)}^{1/4})
  \widehat{f}(t,\xi-\sigma)
  \widehat{\overline{f}}(t,\sigma+\eta-\xi) \partial_{\eta_2}\widehat{f}(t,\xi-\eta) \, \ud \sigma \ud \eta,
\\
M_4(t,\xi) & := \int e^{it \varphi(\eta,\sigma)} {|\eta|}^{-2}\eta_1\big[ \chi(\eta {(1+t)}^{1/4}) - 1 \big]
  \widehat{f}(t,\xi-\sigma) 
  \widehat{\overline{f}}(t,\sigma+\eta-\xi) \partial_{\eta_2}\widehat{f}(t,\xi-\eta) \, \ud \sigma \ud \eta.
\end{split}
\end{align*}

We then estimate
\begin{align*}
| M_3(t,\xi) | & \lesssim {\| \widehat{f}(t) \|}_{L^2} {\| \widehat{f}(t) \|}_{L^2} 
  \int {|\eta|}^{-1} \chi(\eta {(1+t)}^{1/4}) |\partial_2\widehat{f}(t,\xi-\eta)|\,\ud \eta
\\
& \lesssim {\| f(t) \|}_{L^2}^2 {\| \partial_2 \widehat{f}(t) \|}_{L^6} {(1+t)}^{-1/6}
\lesssim B^{2} \e_1^2 {\| (1+{|x|}^2) f(t) \|}_{L^2} {(1+t)}^{-1/6}
\\
& \lesssim B^{3} \e_1^3 {(1+t)}^{-1/6+p_0}.
\end{align*}
The second term is bounded as follows:
\begin{align*}
| M_4(t,\xi) | & \lesssim {\Big\| P_{\geq {(1+t)}^{-1/4}} \partial_1\lap^{-1} \big( \phi(t) \overline{\phi}(t) \big) \Big\|}_{L^2} 
  {\| \partial_2 \widehat{f}(t) \|}_{L^2} 
\\
& \lesssim  {(1+t)}^{1/4} {\| \phi(t) \|}_{L^2} {\|\phi(t)\|}_{L^\infty} {\| x_2 f(t) \|}_{L^2} \lesssim B^{3} \e_1^3 {(1+t)}^{-3/4+p_0}.
\end{align*}
The term $N_4$ can be estimated identically.
We can then conclude that $|N_j| \lesssim {(1+t)}^{-1/8}$, for all $j=1,\dots,4$.
In view of \eqref{NIBP} we obtain the desired bound \eqref{boundhatNR} for $\widehat{\mathcal{N}}(t)$.


\subsection{Estimates for the quintic terms $\calR$}
Under our apriori assumptions we now want to prove:
\begin{align}
\label{quintic1}
& \Big| \mathcal{F} \Big( \lap^{-1} \big(\rd_{1} (A_2 {|\phi|}^2) - \rd_{2} (A_1 {|\phi|}^2) \big) \phi \Big) \Big| \lesssim B^{5}\e_1^5 {(1+t)}^{-11/10},
\\
\label{quintic2}
& \big| \mathcal{F}  \big( A^{\ell} A_{\ell} \phi \big) \big| \lesssim B^{5} \e_1^5 {(1+t)}^{-11/10} \,.
\end{align}

Let us start by estimating the first contribution in the right-hand side of \eqref{quintic1}:
\begin{align*}
\Big| \mathcal{F} \Big( \lap^{-1} \rd_{1} ( A_2 {|\phi|}^2 ) \, \phi \Big) \Big| 
  & \lesssim \Big| \mathcal{F} \Big(\lap^{-1} \rd_{1} \big( A_2 {|\phi|}^2\big) \Big)(\xi) \ast \widehat{\phi}(\xi) \Big| 
\\
& \lesssim  {\Big\| \mathcal{F} \Big(\lap^{-1} \rd_{1} \big(A_2 {|\phi|}^2\big) \Big) \Big\|}_{L^{6/5}}   {\|\widehat{\phi}(\xi) \|}_{L^6}
\\
& \lesssim \big[ I(t) + II(t) \big] B \e_1 {(1+t)}^{p_0},
\end{align*}
having defined
\begin{align}
I(t) & = {\Big\| \mathcal{F} \Big(\lap^{-1} \rd_{1} \big(A_2(t) {|\phi(t)|}^2\big) \Big) \Big\|}_{L^{6/5}(|\xi|\geq 1)}
\\
II(t) & = {\Big\| \mathcal{F} \Big(\lap^{-1} \rd_{1} \big(A_2(t) {|\phi(t)|}^2\big) \Big) \Big\|}_{L^{6/5}(|\xi|\leq 1)}.
\end{align}
Using an $L^3 \times L^2$ H\"older's inequality we can bound
\begin{align*}
 I(t) & = {\Big\| \xi_1{|\xi|}^{-2} \mathcal{F} \big(A_2(t) {|\phi(t)|}^2\big) \Big\|}_{L^{6/5}(|\xi|\geq 1)}
  \lesssim {\big\| {|\xi|}^{-1} \big\|}_{L^{3}(|\xi|\geq 1)} {\big\| \mathcal{F} \big(A_2(t) {|\phi(t)|}^2\big) \big\|}_{L^2}
\\
& \lesssim {\| A_2(t) \|}_{L^\infty} {\|\phi(t)\|}_{L^2} {\|\phi(t)\|}_{L^\infty} \lesssim B^{4} \e_1^4 {(1+t)}^{-2}.
\end{align*}
To estimate the contribution coming from low frequencies we use again H\"older
followed by the Hausdorff-Young inequality:
\begin{align*}
II(t) & = {\Big\| \xi_1{|\xi|}^{-2} \mathcal{F} \big(A_2(t) {|\phi(t)|}^2\big) \Big\|}_{L^{6/5}(|\xi|\leq 1)} 
\lesssim {\big\| {|\xi|}^{-1} \big\|}_{L^{3/2}(|\xi|\leq 1)} {\Big\| \mathcal{F} \big(A_2(t) {|\phi(t)|}^2\big) \Big\|}_{L^6} 
\\
& \lesssim {\Big\| A_2(t) {|\phi(t)|}^2 \Big\|}_{L^{6/5}}
  \lesssim {\| A_2(t) \|}_{L^\infty} {\|\phi(t)\|}_{L^2} {\|\phi(t)\|}_{L^3} \lesssim B^{4} \e_1^4 {(1+t)}^{-4/3}.
\end{align*}
This shows that
\begin{align*}
 \Big| \mathcal{F} \Big( \lap^{-1} \rd_{1} ( A_2 {|\phi|}^2 ) \, \phi \Big) \Big| \lesssim B^{5} \e_1^5 {(1+t)}^{-11/10}.
\end{align*}
Since an identical estimate can be obtained if we exchange the indices $1$ and $2$, we have shown that \eqref{quintic1} holds.

To prove \eqref{quintic2} we first bound
\begin{align*}
\big| \mathcal{F}  \big( A^{\ell} A_{\ell} \phi \big) \big| = \big| \widehat{A^{\ell}} \ast \widehat{A_{\ell}} \ast \widehat{\phi} \big|
  \lesssim {\| \widehat{A^{\ell}} \|}_{L^1} {\| \widehat{A_{\ell}} \|}_{L^1} {\| \widehat{\phi} \|}_{L^\infty}.
\end{align*}
To obtain the desired bound it then suffices to show
\begin{align}
\label{quinticA_l}
{\| \widehat{A_\ell}(t) \|}_{L^1} \lesssim B^{2} \e_1^2 {(1+t)}^{-2/3}.
\end{align}
Since the two cases $\ell=1,2$ are identical we only look at $\ell=2$.
We can estimate
\begin{align*}
{\| \widehat{A_2}(t) \|}_{L^1} \aleq {\big\| \mathcal{F} \big(\lap^{-1} \rd_{1} {|\phi|}^2 \big)(t) \big\|}_{L^1} \lesssim III(t) + IV(t),
\end{align*}
where
\begin{align}
III(t) & = {\big\| {|\xi|}^{-2} \mathcal{F} \big(\partial_1 {|\phi|}^2 \big)(t) \big\|}_{L^1(|\xi|\geq 1)},
\\
IV(t) & = {\big\| \xi_1{|\xi|}^{-2} \mathcal{F}{|\phi|}^2(t) \big\|}_{L^1(|\xi|\leq 1)}.
\end{align}
It is clear that
\begin{align*}
III(t) & \lesssim {\big\| \mathcal{F} \big(\partial_1 {|\phi(t)|}^2 \big) \big\|}_{L^2}
\lesssim {\| \phi(t) \|}_{H^1} {\| \phi(t) \|}_{L^\infty} \lesssim B^{2} \e_1^2 {(1+t)}^{-1},
\end{align*}
which is more than sufficient.
Furthermore, using H\"older's inequality we have
\begin{align*}
IV(t) & \lesssim {\big\| {|\xi|}^{-1} \mathcal{F}{|\phi(t)|}^2 \big\|}_{L^1(|\xi|\leq 1)}
  \lesssim {\big\| \mathcal{F}{|\phi(t)|}^2 \big\|}_{L^3}
  \lesssim {\| \phi(t)^2 \|}_{L^{3/2}} \lesssim B^{2} \e_1^2 {(1+t)}^{-2/3}.
\end{align*}
This gives us \eqref{quinticA_l} and concludes the proof of \eqref{quintic2}.

\subsection{Estimate for the cubic term $\calT$}
Finally, we shall establish \eqref{boundhatT}. We begin by estimating
\begin{align*}
	\abs{\calF( \abs{\phi}^{2} \phi )(t,\xi)} & = \abs{\calF( e^{-it\lap} \abs{\phi}^{2} \phi )(t,\xi)} \aleq 
	{\big\| \calF( e^{-it\lap} \abs{\phi}^{2} \phi )(t) \big\|}_{H^2_\xi} 
\\
	& \aleq {\big\| \abs{\phi}^{2} \phi (t) \big\|}_{L^2} + {\big\| x^2 e^{-it\lap} (\abs{\phi}^{2} \phi)(t) \big\|}_{L^2} . 
\end{align*}
By an $L^2\times L^\infty \times L^\infty$ estimate, the first summand above is easily seen to satisfy a bound of the form $B^3 \e_1^3 {(1+t)}^{-2}$.
For the second one, we use the fact that $x e^{-it\lap} = e^{-it\lap} J $, and the Leibniz rule \eqref{eq:comm4cubic:2} for $J$, to see that
\begin{align*}
  {\big\| x^2 e^{-it\lap} (\abs{\phi}^{2} \phi) \big\|}_{L^2} = {\big\| J^{(2)} (\abs{\phi}^{2} \phi) \big\|}_{L^2} 
  & \aleq
  {\big\| J^{(2)} \phi \big\|}_{L^2} {\| \phi \|}_{L^\infty}^2 + {\big\| J \phi \big\|}_{L^4}^2 {\| \phi \|}_{L^\infty}
  \\& \aleq B^3 \e_1^3 {(1+t)}^{-2} \log^2(2+t) ,
\end{align*}
having used the Gagliardo-Nirenberg inequality \eqref{estlem:GN4J} with $A = 0$
and the apriori bounds \eqref{eq:btstrp:decay}, \eqref{eq:btstrp:ord:2} and \eqref{eq:btstrp:ord:3} in the last inequality.
This completes the proof of \eqref{boundhatT}. \qedhere




\end{proof}

\renewcommand{\bibliofont}{\normalsize}


\end{document}